\documentclass[reqno, 11pt, a4paper]{amsart}
\usepackage{amsmath, amssymb, amsthm}  
\usepackage{graphicx, color}
\usepackage{mathtools}
\usepackage{mathrsfs, enumerate}
\usepackage{hyperref}
\usepackage[truedimen,top=3truecm, bottom=3truecm, left=2.5truecm, right=2.5truecm, includefoot]{geometry}

\newcommand{\ms}[1]{{\color{cyan} #1}}

\usepackage[normalem]{ulem}

\usepackage{tikz}
\usetikzlibrary{shapes.misc}
\usetikzlibrary{arrows.meta}
\tikzset{cross/.style={cross out, draw=black, minimum size=2*(#1-\pgflinewidth), inner sep=0pt, outer sep=0pt},
cross/.default={0.2cm}}
\definecolor{blue-violet}{rgb}{0.54, 0.17, 0.89}
\usetikzlibrary{decorations.pathreplacing}
\usetikzlibrary{patterns}
\begin{document}

\title[Gradient Condition for Asymmetric PEP and Their Scaling Limits]{Characterization of Gradient Condition for Asymmetric Partial Exclusion Processes and Their Scaling Limits}
%[Derivation of the SBE from weakly asymmetric PEPs]{Derivation of the SBE from weakly asymmetric partial exclusion processes}

\author[P. Gon\c{c}alves]{Patr\'{i}cia Gon\c{c}alves}
\address{Departamento de Matem\'atica, Instituto Superior T\'ecnico, Universidade de Lisboa, Av. Rovisco Pais, no. 1, 1049-001 Lisboa, Portugal.}
\email{pgoncalves@tecnico.ulisboa.pt}

\author[K. Hayashi]{Kohei Hayashi}
\address{RIKEN Interdisciplinary Theoretical and Mathematical Sciences Program, 2-1 Hirosawa, Wako, Saitama 351-0198 Japan}
%\email{kohei.hayashi.vh@riken.jp}
\email{koheihayashi0826@gmail.com}

\author[M. Sasada]{Makiko Sasada}
\address{Graduate School of Mathematical Sciences, The University of Tokyo, 3-8-1 Komaba, Meguro-ku, Tokyo 153-8914, Japan}
\email{sasada@ms.u-tokyo.ac.jp}

%\author[G.M. Sch\"{u}tz]{Gunter M. Sch\"{u}tz}
%\address{Departament de Math\'ematica, Instituto Superior Técnico, Universidade de Lisboa, Av. Rovisco Pais, 1049-001 Lisboa, Portugal.}
%\email{gunter.schuetz@tecnico.ulisboa.pt}
%\address{Institute of Complex Systems II, Forschungszentrum J\"{u}lich, 52425 J\"{u}lich, Germany\and 
%Interdisziplin\"{a}res Zentrum f\"{u}r Komplexe Systeme, Universit\"{a}t Bonn, Br\"{u}hler Str. 7, 53119 Bonn, Germany.}
%Center for Mathematical Analysis, Geometry and Dynamical Systems, Instituto Superior Técnico, Universidade de Lisboa, Av. Rovisco Pais, 1049-001 Lisboa, Portugal.}
%\email{g.schuetz@fz-juelich.de}

\begin{abstract}
We consider partial exclusion processes~(PEPs) on the one-dimensional square lattice, that is, a system of interacting particles where each particle random walks according to a jump rate satisfying an exclusion rule that allows up to a certain number of particles can exist on each site.  
Particularly, we assume that the jump rate is given as a product of two functions depending on occupation variables on the original and target sites. 
Our interest is to study the limiting behavior, especially to derive some macroscopic PDEs by means of (fluctuating) hydrodynamics, of fluctuation fields associated with PEPs, starting from an invariant measure. 
The so-called gradient condition, meaning that the symmetric part of the instantaneous current is written in a gradient form, and that the invariant measures are given as a product measure is technically crucial. 
Our first main result is to clarify the relationship between these two conditions, and we show that the gradient condition and the existence of product invariant measures are mutually equivalent, provided the jump rate is given in the above simple form, as it is imposed in most of the literature, and the dynamics is asymmetric.  
Moreover, when the width of the lattice tends to zero and the process is accelerated in diffusive time-scaling, we show that the family of fluctuation fields converges to the stationary energy solution of the stochastic Burgers equation (SBE), under the setting that the jump rate to the right neighboring site is a bit larger than the one to the left side, of which discrepancy is given as square root of the width of the underlying lattice.
This fills the gap at the level of universality of SBE since it has been proved for the exclusion process (a special case of PEP) and for the zero-range process. 
\end{abstract}

\maketitle

\theoremstyle{plain}
\newtheorem{theorem}{Theorem}[section] 
\newtheorem{lemma}[theorem]{Lemma}
\newtheorem{corollary}[theorem]{Corollary}
\newtheorem{proposition}[theorem]{Proposition}
\newtheorem{conjecture}[theorem]{Conjecture}

\theoremstyle{definition}
\newtheorem{definition}[theorem]{Definition}
\newtheorem{remark}[theorem]{Remark}
\newtheorem{assumption}[theorem]{Assumption}
\newtheorem{example}[theorem]{Example}

\makeatletter
\renewcommand{\theequation}{%
\thesection.\arabic{equation}}
\@addtoreset{equation}{section}
\makeatother

\makeatletter
\renewcommand{\p@enumi}{A}
\makeatother

\section{Introduction}
The Kardar-Parisi-Zhang equation, or the KPZ equation in short, is a stochastic partial equation of the following standard form: 
\begin{equation*}
\partial_t h 
= \Delta h 
+ (\nabla h)^2 
+ \sqrt{2}\dot{W}, 
\end{equation*}
where $h=h(t,x)$, $(t,x)\in [0,\infty)\times \mathbb R$ is an unknown object and $\dot{W}=\dot{W}(t,x)$ is the space-time white-noise. 
Throughout the present paper, we consider the one-dimensional line as the domain of the space variable. 
As a similar object, the following SBE is of interest.
\begin{equation*}
\partial_t u = 
\Delta u 
+ \nabla u^2 + \sqrt{2} \nabla \dot{W} .
\end{equation*}
The SBE, at least formally, is obtained by differentiating both sides of the KPZ equation and setting $u=\nabla h$. 
At first glimpse, the equation is ill-posed since the solution should be understood in the distributional sense, and then the nonlinear term does not make sense in an usual way. 
However, this has been overcome in several ways~\cite{hairer2013solving, gubinelli2017kpz, gonccalves2014nonlinear}, with a generalization which is applicable for various types of singular SPDEs~\cite{hairer2014theory, gubinelli2015paracontrolled}. 
In particular, the paper \cite{gonccalves2014nonlinear} introduced the notion of stationary energy solution to the KPZ/SB equation as a martingale problem, of which uniqueness is shown in \cite{gubinelli2018energy}, and this established a robust way to derive the KPZ/SB equation through scaling limits of microscopic systems.

On the other hand, what is physically intriguing for the KPZ/SB equation is its universality. 
Indeed, up to now, the KPZ/SB equation has been derived from many types of microscopic models \cite{bertini1997stochastic, gonccalves2015stochastic, diehl2017kardar, jara2019scaling, ahmed2022microscopic, cannizzaro2023abc} on the one-dimensional square lattice, under the so-called \textit{weakly asymmetric} regime. 
More precisely, density fluctuation fields associated to each model, when the width of the lattice, say $\varepsilon>0$, goes to zero and the process is accelerated by the diffusive time scaling $\varepsilon^{-2}$, and the jump rates are weakly asymmetric in the sense that the rate to one direction is larger than that to the other direction and the order of this discrepancy is $\varepsilon^{1/2}$, converge to the stationary solution of the SBE. 
Note here that there are some other works concerning the derivation of the KPZ equation from totally asymmetric models \cite{jara2020stationary, hayashi2022derivation, hayashi2023derivation, goncalves2023derivation}, for which cases the nonlinear term of the limiting SPDE is extracted by some perturbation argument for a driving potential of each model. 
This argument physically corresponds to forcing the temperature of the system to be divergent as a scaling parameter goes to infinity.

An aim of the present paper is to give another example of a microscopic model from which the derivation of the KPZ/SB equation is possible by means of fluctuating hydrodynamics. 
In order to establish equilibrium fluctuations, and also hydrodynamic limits, for interacting particle systems, the so-called \textit{gradient condition} is known to be technically significant. 
This condition means that the symmetric part of the instantaneous current can be written in a gradient form, which enables us to conduct a summation by parts twice, and thus gives rise to the viscosity term at the macroscopic equation. 
Moreover, regarding the study of equilibrium fluctuation, the existence of product invariant measures is also crucial and in most of the previous models this is indeed the case and the dynamics is initialized from these product invariant measures. 

In the present paper, we study equilibrium fluctuations of partial exclusion processes~(PEPs), which is an interacting particle system, where each particle jumps, on the one-dimensional square lattice, to nearest-neighbor sites with asymmetric rates.  
PEP is a generalization of simple exclusion process~(SEP) in the sense that a finite number of particles, say $\kappa$, can coexist at the same site, whereas for SEP, more than two particles cannot live on the same site, and to clarify this relation, the PEP is also referred to as $\kappa$-exclusion process. 
Moreover, we assume, as most of the previous results do, that the jump rate is given by the product of two functions, one of which depends only on the number of particles of the departing site, whereas the other function depends on that of the target site. 
For this situation, surprisingly, we show, as our first main result, that for asymmetric PEP the existence of product invariant measures is equivalent to the gradient condition.  
In particular, we can give an alternative example of asymmetric, gradient systems, of which invariant measure takes a product form. 
Then, following a robust approach established in \cite{gonccalves2014nonlinear}, we derive the SBE from density fluctuation fields associated to our asymmetric PEP, imposing the weak asymmetry.   
This is the second ingredient of the present paper.

\subsection*{Organization of the Paper}
In Section \ref{sec:pep_model}, we give a precise description of PEP and state the main results. 
First, we characterize the gradient condition of PEP, and then, under the gradient condition, state that the density fluctuations of PEP, in the weakly asymmetric regime, converge to the stationary energy solution of the SBE. 
In Section \ref{sec:pep_gradient}, we show the first result concerning the characterization of the gradient condition, and as well, we give a formula for the diffusion coefficient in the limiting equation. 
The forthcoming sections are devoted to proving the derivation of the SBE. 
First, in Section \ref{sec:pep_estimate}, we give some dynamical estimates which will be used to show the main theorem. 
These estimates include the $\mathscr H^{-1}$-norm estimate, which is referred to as the Kipnis-Varadhan estimate, and Boltzmann-Gibbs principle, which enables us to replace a local function of occupation variables by local averages. 
Next, in Section \ref{sec:pep_outline}, we give an outline of the proof, starting from a martingale decomposition of the density fluctuation field.  
In particular, each field appearing in the martingale decomposition is dealt with the Boltzmann-Gibbs principle, and see that the anti-symmetric part of the martingale decomposition gives rise to the nonlinear term of the SBE in the limit. 
Finally, in Section \ref{sec:pep_tightness} and \ref{sec:pep_identification}, we show the tightness of each term in the martingale decomposition, respectively.

\subsection*{Notation}
%Given two real-valued functions $f$ and $g$ depending on a variable $u\in\mathbb{R}^d$ we will write $f(u)\le C g(u)$ if there exists a constant $C>0$ such that $f(u)\le C g(u)$ for any $u$. 
%Moreover, we write $f=O(g)$ (resp. $f=o(g)$) in the neighborhood of $u_0$ if $|f|\le |g|$ in the neighborhood of $u_0$ (resp. $\lim_{u\to u_0}f(u)/g(u)=0$). 
%Sometimes it will be convenient to make precise the dependence of the constant $C$ on some extra parameters and this will be done by the standard notation $C(\lambda)$ if $\lambda$ is the extra parameter. 
When there is no confusion, in several estimates, we use the same letter $C$ as a positive constant which might be different from line to line. 
We denote by $\langle \cdot, \cdot \rangle_{L^2(\mathbb{R})}$  the inner product in $L^2(\mathbb{R})$, i.e. for any $f,g\in L^2(\mathbb{R})$ 
\begin{equation*}
\langle f, g\rangle_{L^2(\mathbb{R})}
\coloneqq \int_{\mathbb{R}} f(x) g(x) dx, 
\end{equation*}
 and by $\| \cdot\|_{L^2(\mathbb{R})}$  the $L^2(\mathbb{R})$-norm, i.e. 
$
\|f \|_{L^2(\mathbb{R})}
\coloneqq ( \langle f, f\rangle_{L^2(\mathbb{R})} )^{1/2}$.
%For each real sequence $g =(g_x)_{x\in\mathbb Z}$, we define a shift operator $T^-$ by $T^-_a g_x = g_{x-a}$ for any $a\in\mathbb R$, whereas 
Let $\tau_x$ be a canonical shift: $\tau_x g_y = g_{x+y}$ for any $x,y\in\mathbb Z$. 
Moreover, define the following discrete derivative operators:
\begin{equation*}
\nabla^n g_x = n(g_{x+1}-g_x),\quad
\Delta^n g_x = n^2 (g_{x+1} +g_{x-1} - 2g_x) .
\end{equation*}
Let $\mathcal{S}(\mathbb R)$ be the space of Schwartz functions and $\mathcal{S}'(\mathbb R)$ be its dual, i.e., the set of real-valued linear continuous functionals defined on $\mathcal{S}(\mathbb R)$. 
Let $C([0,T],\mathcal{S}'(\mathbb R))$ be the space of $\mathcal{S}'(\mathbb R)$-valued continuous functions on $[0,T]$ endowed with the uniform topology whereas let $D([0,T],\mathcal{S}'(\mathbb R))$ be the space of $\mathcal{S}'(\mathbb R)$-valued c\`adl\`ag (right-continuous and with left limits) functions on $[0,T]$ endowed with the Skorohod topology.

\section{Model and Result}
\label{sec:pep_model}
We consider the following one-dimensional asymmetric partial exclusion processes with nearest-neighbor interactions. 
Fix $\kappa\in\mathbb N$ and let $\mathscr X=\{0,1,\ldots, \kappa \}^{\mathbb Z}$ be the configuration space, and we denote each element of $\mathscr X$ by $\eta=\{\eta_x\}_{x\in\mathbb Z}$ where $\eta_x$ is the number of particles at the site $x\in\mathbb Z$. 
Let $n>0$ be the scaling parameter which will be sent to infinity and let $L_n$ be the operator defined on any local\footnote{We say that a function on the configuration space $\mathscr X$ is  \textit{local} if it depends only on a finite number of occupation variables, i.e.  if there exists a finite set $A \subset \mathbb Z$ such that $f (\eta) = f (\tilde\eta )$ for any
$\eta,\tilde\eta\in\Omega$ with $\eta(x) = \tilde \eta(x)$ for any $x \in A$ and the support of $f$, denoted by $\mathrm{supp}(f)$, is the smallest
of those sets.} function $f:\mathscr X\to \mathbb R$ by
\begin{equation}
\label{eq:generator_def}
L_nf(\eta)
%= \frac{n^2}{2} \sum_{\substack{x,y\in\mathbb Z, \\|x-y|\le 1} }
%r^n_{x,y}(\eta) \nabla_{x,y} f(\eta) 
=n^2\sum_{x\in\mathbb Z} p_n r_{x,x+1}(\eta) \nabla_{x,x+1}f(\eta)
+ n^2\sum_{x\in\mathbb Z} q_n r_{x,x-1}(\eta)\nabla_{x,x-1}f(\eta) 
\end{equation}
where $p_n,q_n\ge0$ are transition probabilities satisfying 
\begin{equation}
\label{eq:condition_tp}
p_n + q_n = 1. 
\end{equation}   
Above, for any $x,y \in \mathbb Z$ we set $\nabla_{x,y}f(\eta)= f(\eta^{x,y})-f(\eta)$ and $\eta^{x,y}$ is the configuration obtained after a particle jumps from $x$ to $y$: 
\begin{equation*}
\eta^{x,y}_z= 
\begin{cases}
\begin{aligned}
&\eta_x-1  &&\text{ if $z=x$,}\\
&\eta_y+1  &&\text{ if $z=y$,} \\
&\eta_z &&\text{ otherwise.}
\end{aligned}
\end{cases}
\end{equation*}
Throughout this paper, we assume that the jump rate takes the form 
\begin{equation}
\label{eq:jump_rate_assumption}
r_{x,y}(\eta) 
= r(\eta_x ,\eta_y)
= c(\eta_x)  d(\kappa - \eta_y)
\end{equation}
with some $c,d: \{0,1,\dots, \kappa\} \to [0,\infty)$ such that $c(0) = d(0)=0$, whereas $c(m)>0$ and $d(m)>0$ for $m \neq 0$. 
Without loss of generality, we may assume the following normalizing condition: 
\begin{equation}
\label{eq:jump_rate_normalization}
c(\kappa)=d(\kappa).
\end{equation}
Indeed, replacing $c(\cdot)$ by $c(\cdot)\sqrt{d(\kappa)}/\sqrt{c(\kappa)}$ and $d(\cdot)$ by $d(\cdot)\sqrt{c(\kappa)}/\sqrt{d(\kappa)}$ respectively, we obtain the same jump rate $r$. 
In particular, this normalization allows us to determine the pair $(c,d)$ uniquely from the jump rate $r$ given by the form \eqref{eq:jump_rate_assumption}.
Note that the jump rate $r$ clearly satisfies the following ellipticity condition: there exists a constant $\varepsilon_0 >0$ such that 
\begin{equation}
\label{eq:jump_rate_ellipticity}
\varepsilon_0 \le r(\eta_x ,\eta_{x+1}) 
\le 1/\varepsilon_0
\end{equation}
for any $\eta\in\mathscr X$ and $x\in\mathbb Z$ such that $\eta_x>0$ and $\eta_{x+1}<\kappa$.  
%The necessary and sufficient condition in order that invariant measures are of product form is given in \cite[Propostion 3.6]{fajfrova2016invariant}.
%\ms{It would be better to use the representation $r(\eta_x,\eta_y)=a c(\eta_x)d(\kappa - \eta_y) $ for some $a >0$, $c,d : \{0,1,\dots,\kappa\} \to [0,\infty)$ such that $c(0)=d(0)=0, c(\kappa)=d(\kappa)=1$ and $c(m)>0, d(m)>0$ for all $m \neq 0$. We can always rewrite the original jump rate into this form by taking $a:=c(\kappa)d(\kappa)$ and $c(\cdot) \to \frac{c(\cdot)}{c(\kappa)}$ and $d(\cdot) \to \frac{d(\cdot)}{d(\kappa)}$, and in this representation, $r$ and $(a,c,d)$ is one-to-one with $a=r(\kappa,0), d(m)=\frac{r(\kappa,\kappa-m)}{a}, c(m)=\frac{r(m,0)}{a}$.}\\
%\pg{(comments by Patricia) better to put $m!$ in the numerator to clarify the fact that the invariant measure matches the binomial distribution when the functions $c$ and $d$ are both identity.}

%%%%%%%%%%%%%%%%%%%%%%%%%%%%%%%%%%%%%%%%%%%%%%%%%%%%%%%%%%%%
\begin{figure}[htpb]%[h!]
\centering
\begin{tikzpicture}[scale=0.90]
\node at (4,0.7) (a) [fill,circle,inner sep=0.15cm,color=brown] { };
\node at (5,0.7) (a) [fill,circle,inner sep=0.15cm,color=brown] { };
\node at (5,1.3) (a) [fill,circle,inner sep=0.15cm,color=brown] { };
\node at (6,0.7) (b) [fill,circle,inner sep=0.15cm,color=brown] { };
\node at (6,1.3) (b') [fill,circle,inner sep=0.15cm,color=brown] { };
\node at (6,1.8) (b*) [fill,circle,inner sep=0.05cm,color=gray] { };
\node at (6,2) (b'') [fill,circle,inner sep=0.05cm,color=gray] { };
\node at (6,2.2) (b**) [fill,circle,inner sep=0.05cm,color=gray] { };
\node at (6,2.7) (o) [fill,circle,inner sep=0.15cm,color=brown] { };
\draw (0,0) -- (14,0);
%\foreach \x in {0,...,2} \draw (\x, 0.3) -- (\x, -0.3) node (\x) [above=1cm] { } node[below] {\x};
%\draw (0, 0.3) -- (0, -0.3); % node (0) [above=1cm] { } node[below] { };
\draw (1, 0.3) -- (1, -0.3); % node (0) [above=1cm] { } node[below] { };
\draw (2, 0.3) -- (2, -0.3); % node (0) [above=1cm] { } node[below] { };
\node at (0,0) (A) [below = 1.8cm] { };
\node at (1,0) (A') [below = 1.2cm] { };
\node at (3,-0.6) {\ldots};
\foreach \x\y in {4/$x-1$,6/$x+1$} \draw (\x, 0.3) -- (\x, -0.3) node (\x) [above=1.4cm] { } node [below] {\y};
\foreach \x\y in {5/$x$} \draw (\x, 0.3) -- (\x, -0.3) node (\x) [above=1.4cm] { } node [below=0.065cm] {\y};
\node at (7,-0.6) {\ldots};
\draw[-latex] (5) to[out=100,in=120] node[midway,font=\scriptsize,above] { } (o);
\draw (5.1,2.9) node[cross,brown] { };
\draw[-latex] (5) to[out=130,in=80] node[midway,above] {$\hspace{-1.9cm}q_nr_{x,x-1}$} (4);
\foreach \z\w in {8/$y-1$,10/$y+1$} \draw (\z, 0.3) -- (\z, -0.3) node (\z) [above=1.4cm] { } node [below] {\w};
\foreach \z\w in {9/$y$} \draw (\z, 0.3) -- (\z, -0.3) node (\z) [above=1.4cm] { } node [below=0.065cm] {\w};
\node at (8,0.7)  [fill,circle,inner sep=0.15cm,color=brown] { };
\node at (8,1.3) [fill,circle,inner sep=0.15cm,color=brown] { };
%\node at (8,1.2) [fill,circle,inner sep=0.05cm,color=gray] { };
%\node at (8,1.4) [fill,circle,inner sep=0.05cm,color=gray] { };
%\node at (8,1.6) [fill,circle,inner sep=0.05cm,color=gray] { };
%\node at (8,2.1) [fill,circle,inner sep=0.15cm,color=brown] { };
\node (c) at (9,0.7) [fill,circle,inner sep=0.15cm,color=brown] { };
%\node (c) at (9,1.3) [fill,circle,inner sep=0.15cm,color=brown] { };
%\node (c) at (9,1.2) [fill,circle,inner sep=0.05cm,color=gray] { };
%\node (c) at (9,1.4) [fill,circle,inner sep=0.05cm,color=gray] { };
%\node (c) at (9,1.6) [fill,circle,inner sep=0.05cm,color=gray] { };
\node (c) at (9,1.3) [fill,circle,inner sep=0.15cm,color=brown] { };
%\node (c1) at (9,2.1) [fill,circle,inner sep=0.15cm,color=brown] { };
\node (c1) at (9,1.9) [fill,circle,inner sep=0.15cm,color=brown] { };
\node at (9,1.9) (Z) [above] { };
\node at (8,1.9) (W) [above] { };
\node at (10,1.9) (Y) [above] { };
\draw[-latex] (Z) to[out=130,in=80] node[midway,above]  {$\hspace{-0.5cm}q_nr_{y,y-1}$} (W);
\draw [decorate,decoration={brace,amplitude=0.1cm, mirror}] (6.5,0.4) -- (6.5,3) node [black,midway,xshift=0.3cm] {$\kappa$};
\draw[-latex] (Z) to[out=60,in=120] node[midway,above] {$\hspace{1.0cm}p_nr_{y,y+1}$} (Y);
%\foreach \x\y in {12/$n-2$,13/$n-1$,14/$n$} \draw (\x, 0.3) -- (\x, -0.3) node (\x) [above=1cm] { } node [below] {\y};
\draw (12, 0.3) -- (12, -0.3); 
\draw (13, 0.3) -- (13, -0.3); 
\node at (12,0.7)  [fill,circle,inner sep=0.15cm,color=brown] { };
\node at (13,0.7)  [fill,circle,inner sep=0.15cm,color=brown] { };
\node at (13,1.3)  [fill,circle,inner sep=0.15cm,color=brown] { };
\node at (11,-0.6) {\ldots};
\node at (13,0) (B') [below = 1.2cm] { };
\node at (14,0) (B) [below = 1.8cm] { };
\end{tikzpicture}
\vspace*{-1cm}
\caption{The dynamics of the PEP.}
\label{fig:pep_dynamics}
\end{figure}
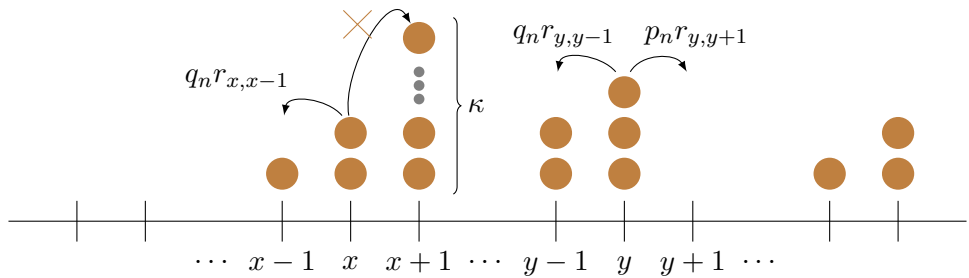

In what follows, let $\eta=\{ \eta_x(t): t \ge 0, x\in \mathbb Z\}$ be the Markov process on $\mathscr X$ generated by $L_n$ where we omit the dependency on $n$ by abuse of notation.  
The existence of the dynamics of PEP in infinite volume can be proved by \cite[Theorem 1.3.9]{liggett1985interacting}. 
Schematic description of the dynamics is shown in Figure~\ref{fig:pep_dynamics}. 
%Note that under the weak asymmetry \eqref{eq:weak_asymmetry} the jump rate $r_{x,x+1}$ depends on $n$ and so does the process $\eta$, but we omit this dependency by abuse on notation. 

Next, we define product measures associated with the process $\eta$. 
For each $\lambda > 0$, consider the following product measure on $\mathscr X$ whose marginal is given by  
\begin{equation*}
\overline \nu_\lambda(\eta_x=m)
= \frac{1}{Z_\lambda}
\frac{\lambda^m}{c!(m) d!(\kappa-m)} 
%\ \forall x \in \mathbb Z,\quad
\end{equation*}
for each $x\in\mathbb Z$ and $m \in \{0.\ldots,\kappa\}$, where 
\begin{equation*}
Z_\lambda
= \sum_{m=0}^\kappa
\frac{\lambda^m}{c!(m) d!(\kappa -m)}.   
\end{equation*}
Above, $c!(m)=c(m)c(m-1)\cdots c(1)$ for $m \ge 1$ and $d!(\kappa-m)=d(\kappa-m)d(\kappa-m-1)\dots d(1)$ for $\kappa- m \ge 1$ while we use the convention $c!(0)=1$ and $d!(0)=1$. 
When the jump rate $r$ satisfies \eqref{eq:jump_rate_assumption}, we have that  
\begin{equation}\label{eq:reversible}
r(\eta_x,\eta_y) \overline\nu_\lambda(\eta_x)\overline \nu_\lambda(\eta_y)
=r(\eta_y+1,\eta_x-1) \overline\nu_\lambda(\eta_x-1)\overline \nu_\lambda(\eta_y+1)
\end{equation}
for any $\eta_x \ge 1$ and $\eta_y \le \kappa-1$.
In particular, the PEP is reversible with respect to $\bar{\nu}_{\lambda}$ if $p_n=q_n$.
Since $(0,\infty) \to (0,\kappa) : \lambda \to E_{\bar{\nu}_{\lambda}}[\eta_0]$ is a strictly increasing bijection, it has an inverse function.
In the sequel, fix any $\rho \in (0,\kappa)$ and choose $\lambda = \lambda(\rho)$ in such a way that 
\begin{equation*}
E_{\overline\nu_{\lambda(\rho)}}[\eta_0] = \rho, 
\end{equation*}
in which situation we write $\nu_\rho= \overline\nu_{\lambda(\rho)}$. 
In what follows, let 
\begin{equation*}
\chi(\rho)=\mathrm{Var}_{\nu_\rho}[\eta_0]
\end{equation*}
be the static compressibility. 
By the explicit form of the measure, we have 
\begin{equation*}
E_{\overline\nu_{\lambda}}[\eta_0]=\lambda \frac{d}{d\lambda}(\log Z_{\lambda}), \quad \mathrm{Var}_{\overline\nu_{\lambda}}[\eta_0]=\lambda \frac{d}{d\lambda} \left( \lambda \frac{d}{d\lambda}(\log Z_{\lambda})\right)
\end{equation*}
and thus 
\begin{equation}
\label{eq:compressibility}
\chi(\rho)=\lambda(\rho)\left(\frac{d \lambda(\rho)}{d\rho}\right)^{-1}.
\end{equation}
For any local function $f:\mathscr X \to \mathbb{R}$ and $\rho \in (0,\kappa)$, let $\Phi_f$ be defined by  
\begin{equation}
\label{eq:definition_expectation_local_function}
\Phi_f(\rho)
= E_{\nu_{\rho}}[f].
\end{equation}

\subsection{Main Result 1: Characterization of the Gradient Condition} 
Recall from \eqref{eq:generator_def} the definition of the generator $L_n$.
Let $L^*_n$ denote the adjoint operator of $L_n$ with respect to $\nu_\rho$, acting on each local function on $\mathscr X$.  
Then, 
\begin{equation*}
L^*_n
= n^2 \sum_{x\in\mathbb Z} p_n r(\eta_{x+1},\eta_x)
\nabla_{x+1,x}
+ n^2 \sum_{x\in\mathbb Z} q_n r(\eta_{x},\eta_{x+1})
\nabla_{x,x+1}. 
\end{equation*}
Let 
\begin{equation*}
S_n=\frac{L_n+L^*_n}{2}
= \frac{n^2}{2}\sum_{x\in\mathbb Z} 
\big\{ r(\eta_x,\eta_{x+1}) \nabla_{x,x+1}
+ r(\eta_{x+1},\eta_{x}) \nabla_{x+1,x}\big\}
\end{equation*}
and 
\begin{equation*}
A_n=\frac{L_n-L^*_n}{2}
= \frac{n^2(p_n-q_n)}{2}\sum_{x\in\mathbb Z} 
\big\{ r(\eta_x,\eta_{x+1}) \nabla_{x,x+1}
- r(\eta_{x+1},\eta_{x}) \nabla_{x+1,x} \big\}
\end{equation*}
be the symmetric and anti-symmetric parts of the generator $L_n$, respectively. 
Above, we used the condition \eqref{eq:condition_tp} on the symmetric part. 
Let $W^S_{x,x+1}$ (resp. $W^A_{x,x+1}$) denote the instantaneous current of the symmetric (resp. anti-symmetric) part of the generator defined by
\begin{equation}
\label{eq:pep_current_computation_symmetric}
 W^S_{x,x+1}
= \frac{1}{2}\big( r(\eta_x,\eta_{x+1})-r(\eta_{x+1},\eta_x)\big)  
\end{equation}
and 
\begin{equation}
\label{eq:pep_current_computation_anti-symmetric}
 W^A_{x,x+1}
= \frac{p_n-q_n}{2}
\big( r(\eta_x,\eta_{x+1})+ r(\eta_{x+1},\eta_x) \big).
\end{equation}
Note that we have the relation $n^{-2}S_n\eta_x=W^S_{x-1,x}-W^S_{x,x+1}$ and  $n^{-2}A_n\eta_x=W^A_{x-1,x}-W^A_{x,x+1}$.

\begin{definition}
We say that the process $\eta$ is of gradient type, or it is gradient, if there exists some function $h: \{ 0,\ldots , \kappa\} \to \mathbb R$ such that for any $x \in \mathbb Z$, 
\begin{equation*}
W^S_{x,x+1}=h(\eta_x)-h(\eta_{x+1}),
\end{equation*}
otherwise we say that the process is of non-gradient type, or it is non-gradient. 
\end{definition}

Then we can show that the gradient condition for our asymmetric PEP is equivalent to the existence of product invariant measures. Moreover, we can completely characterize the class of jump rates satisfying the gradient condition.

\begin{theorem}
\label{thm:gradient_condition}
%Fix $\rho\in(0,\kappa)$. 
Assume $p_n \neq q_n$ and the jump rate is given by \eqref{eq:jump_rate_assumption}. 
Then the followings are equivalent: 
\begin{itemize}
    \item%[(a)] 
    The product measure $\nu_\rho$ is invariant under the dynamics of PEP for any $\rho \in (0,\kappa)$.
    \item%[(b)] 
    The PEP is of gradient type. 
    \item%[(c)] 
    For any $m\in \{0,\ldots , \kappa\}$, 
%\begin{equation}
%\erasekh{c(\kappa)d(\kappa)=c(\kappa)d(\kappa-m)+c(m)d(\kappa),}
%\end{equation}
%\erasekh{which, under the assumption \eqref{eq:jump_rate_normalization}, reads}
\begin{equation}
\label{eq:gradient_condition}
d(\kappa-m)=c(\kappa)-c(m).
\end{equation}
\end{itemize}
In particular, the above conditions are satisfied if and only if 
\[
r(\eta_x,\eta_y)=c(\eta_x)(c(\kappa)-c(\eta_y))
\]
for some function $c : \{0,1,\dots,\kappa\} \to [0,\infty)$ such that $c(0)=0, c(\kappa)>0$ and $0< c(m) < c(\kappa)$ for any $m \in \{1,2,\dots, \kappa-1\}$. Moreover, if this holds, then the gradient condition holds with $h(\cdot)=(1/2)c(\kappa)c(\cdot)$. 
\end{theorem}

\begin{remark}
Here we give some comments on some papers studying the invariance of product measures. 
In the paper \cite{cocozza1985processus}, misanthrope processes were introduced and the invariant measures for the processes were studied.   
They assume, however, that the jump rates are defined in $\mathbb N\times\mathbb N$ and are increasing in the first variable and decreasing in the second one, which rules out our choice of rates. 
%\kh{asymmetric?}
On the other hand, we should mention here the paper \cite{sethuraman2018hydrodynamic} which studies the hydrodynamic limit for the misanthrope process, relaxing the assumption on the misanthrope process from that in the original paper \cite{cocozza1985processus} and allowing our PEPs to be in the scope. 
They also give a condition that product measures are invariant for the process under the restriction \eqref{eq:jump_rate_assumption}, for which case they call the process \textit{decomposable} misanthrope systems. 
%\erasems{and the condition that they derived is nothing but } 
The condition that they derived also works for decomposable misanthrope systems without exclusive constraint, and we have shown that, for PEPs the condition is simplified as \eqref{eq:gradient_condition}.  
Also, they do not point out the equivalence of the condition to the gradient condition, and our additional contribution is that we unveiled the relation. We also remark that for exclusion processes, this equivalence was proved in \cite[Proposition 5.1]{nagahata1998}.
%\pat{here we should say that besides the paper  of Coccoza  on misanhrope processes does not allow for our rates, since it is assumed that the jump rates are defined in $\mathbb N\times \mathbb N$ and are increasing in the first variable and decreasing in the second one, which rules out our choice of rates, neverthless misanthrope processes can be more general as mentioned in sunder's paper and our rates are there included. agree?}
%Our contribution is that we related the condition to that the PEP is of gradient type. 
\end{remark}

\begin{example}[SEP($\kappa$)] 
When $c(m)=d(m)=m$, the process is of gradient type and the measure $\nu_\rho$ becomes a product binomial distribution: 
\begin{equation*}
\nu_\rho(\eta_x=m)
= %\begin{pmatrix} \kappa \\ m \end{pmatrix}
\frac{\kappa!}{m!(\kappa-m)!} 
%(\rho/\kappa)^m (1-\rho/\kappa)^{\kappa-m}. 
\Big(\frac{\rho}{\kappa}\Big)^m \Big(1-\frac{\rho}{\kappa}\Big)^{\kappa-m}. 
\end{equation*}
In this case, for any $p_n$ and $q_n$, the product measure $\nu_\rho$ is invariant under the dynamics of the PEP. 
In particular, the process with this form of jump rate is called SEP($\kappa$) and its hydrodynamic limit is the linear heat equation.
Moreover, it is pointed out in \cite{schutz1994non} that the stochastic duality holds only for this choice of rates. 
\end{example}

\begin{example}
When $c(m)=d(m)=\mathbf{1}_{m>0}$, we obtain the model studied in \cite[Chapter 7]{kipnis1999scaling} and the process is non-gradient if $\kappa \ge 2$. 
For this jump rate, the product measures we defined above become the following truncated geometric distribution whose common marginal is given by 
\begin{equation*}
\nu_\rho (\eta_x=m)
= \frac{\lambda^m}{1 + \lambda + \cdots +\lambda^\kappa},
\end{equation*}
where $\lambda=\lambda(\rho)$ is chosen in such a way that $E_{\nu_\rho}[\eta_0]= \rho$ for any given $\rho \in (0,\kappa)$.
Moreover, when the process is symmetric in the sense that $p_n = q_n$, this product measure is invariant and reversible for the PEP. 
For the symmetric case, \cite{kipnis1994hydrodynamical} and \cite{sellami1999equilibrium} study the hydrodynamic limit and equilibrium fluctuations, respectively.   
Note, however, that for the asymmetric case $p_n\neq q_n$, the above product measure is not invariant for the PEP. 
Without any knowledge of the invariant measures, only for the totally asymmetric case, the hydrodynamic limit is obtained in \cite{seppalainen1999}, albeit the study of fluctuations is missing.  
\end{example}

\begin{example}
To interpolate the above two models, we can consider a one-parameter family of functions $c(m)=d(m)=(1-\theta)\mathbf{1}_{m>0} + \theta m$ with a parameter $\theta\in [0,1]$, for instance, and obtain the sequences of product measures $\{\nu_{\rho}^{\theta}\}_{\theta}$ interpolating the product binomial distribution and the product truncated geometric distribution.
We may also interpolate them in a different way. 
\end{example}

\begin{remark}
Here we give a comment on some duality relation for the PEP with the jump rate of the form \eqref{eq:jump_rate_assumption}.  
We can see that when the PEP is driven by the quadruple $(p_n,q_n,c,d)$, then the dynamics of ``holes'' $\hat{\eta}(t)\coloneqq(\kappa-\eta_{x}(t))_x$ turns out to be that of the PEP driven by $(q_n,p_n,d,c)$. 
%as $\eta(t)$, namely $\hat{\eta}(t)$ is the Markov process generated by $L_n$. For the jump rate of the form \eqref{eq:jump_rate_assumption}, the PEP has the hole-particle duality if and only if $c(\cdot)=d(\cdot)$.
This property is referred to as the particle-hole duality, particularly when $p_n=q_n$ and $c=d$. 
\end{remark}
%\pat{above should be $\hat{\eta}(t) :=(\kappa-\eta_{x}(t))_x$, agree?}

\subsection{Main Result 2: Scaling Limits for Density Fields}
Next, we state our second main result which is concerned with the time evolution of fluctuation fields associated with our asymmetric PEPs. 
To state the result, here we recall the notion of stationary energy solution, which is introduced in~\cite{gonccalves2014nonlinear}, of the $(1+1)$-dimensional SBE of the following form:
\begin{equation}
\label{eq:SBE_general}
\partial_t u = D \Delta u 
+ \Lambda \nabla u^2 + \sqrt{2\chi D} \nabla \dot{W} 
\end{equation}
where $D,\chi > 0$ and $\Lambda \in \mathbb{R}$ and $\dot{W}$ denotes the one-dimensional space-time white-noise.   
%In what follows, let $\mathcal S(\mathbb R)$ be the space of Schwartz functions and $\mathcal{S}'(\mathbb R)$ its dual, i.e. the set of linear continuous functionals defined on $\mathcal S(\mathbb R)$ and taking real values. 
%\ms{(I erased the above paragraph as it is already given in the notation. Please check.)}
We begin with the definition of stationarity.

\begin{definition}
We say that an $\mathcal{S}^\prime (\mathbb{R})$-valued process $u = \{ u_t  : t \in [0,T] \} $ satisfies the condition \textbf{(S)} if for all $t \in [0,T]$, the random variable $u_t$ has the same distribution as the space white-noise with variance $\chi$. %$D/(2\nu)$. 
\end{definition}

For a process $u = \{ u_t: t \in [0,T]\}$ satisfying the condition \textbf{(S)}, we define 
\begin{equation}
\label{eq:def_quadratic_function_approximation}
\mathcal{A}^\varepsilon_{s,t} (\varphi) 
= \int_s^t \int_{\mathbb{R}} u_r (\iota_\varepsilon (x;\cdot))^2 \nabla \varphi (x) dx dr ,  
\end{equation}
for every $0 \le s < t \le T $, $\varphi \in \mathcal{S} (\mathbb{R} ) $ and $\varepsilon > 0 $. 
Above, the function $\iota_\varepsilon (x ; \cdot ) : \mathbb{R} \to \mathbb{R}$ is given by 
\begin{equation}\label{eq:definition_iota}
\iota_\varepsilon (x;y) 
=\varepsilon^{-1} \mathbf{1}_{[x,x+\varepsilon)}(y) 
\end{equation}
for each $x \in \mathbb{R} $ and $\epsilon>0$. 
Although the function $\iota_\varepsilon(x,\cdot)$ does not belong to the Schwartz space, the quantity \eqref{eq:def_quadratic_function_approximation} is well-defined. To that end it is enough to consider a sequence of smooth compactly supported functions $(\rho_{\varepsilon,k})_k$ with $\rho_{\varepsilon,k}:\mathbb R\to[0,+\infty)$ such that $\|\rho_{\varepsilon,k}\|_2^2\leq 2\|\iota_\varepsilon\|_2^2=\varepsilon^{-1}$ for all $k\in\mathbb N$ and $\displaystyle \lim_{k\to+\infty}  \|\iota_\varepsilon-\rho_{\varepsilon,k}\|_2^2=0$.
Replacing $\iota_\varepsilon$ by $\rho_{\varepsilon,k}$ in $\mathcal{A}^\varepsilon_{s,t} (\varphi) $,  the resulting sequence is  Cauchy  in $L^2$ and we can give a sense to $\mathcal{A}^\varepsilon_{s,t} (\varphi) $ as  the limit of that sequence. For details we refer the reader to, for example~\cite[Section 2.2]{gonccalves2015stochastic}.
\begin{definition}
Let $u = \{ u_t :t \in [0,T]\}$ be a process satisfying the condition \textbf{(S)}. 
We say that the process $u$ satisfies the energy condition \textbf{(EC)} if there exists a constant $K>0$ such that for any $\varphi \in \mathcal{S} (\mathbb{R} )$, any $0 \le s<t\le T$ and any $0 < \delta < \varepsilon < 1 $,  
\begin{equation*}
\mathbb E \big[ \big| \mathcal{A}^\varepsilon_{ s, t } (\varphi ) - \mathcal{A}^\delta_{ s, t } (\varphi ) \big|^2 \big] 
\le K \varepsilon (t-s) \| \nabla \varphi \|^2_{ L^2(\mathbb{R})} .
\end{equation*}
Here $\mathbb{E}$ denotes the expectation with respect to the measure of a probability space where the process $u$ lives. 
\end{definition}
The following result is proved in \cite{gonccalves2014nonlinear}. 

\begin{proposition}
\label{prop:nonlinear}
Assume $\{ u_t:t\in [0,T]\} $ satisfies the conditions \textbf{(S)} and \textbf{(EC)}. Then there exists an $\mathcal{S}^\prime (\mathbb{R} )$-valued process $\{ \mathcal{A}_t : t \in [0, T ] \} $ with continuous trajectories such that  
\begin{equation*}
\mathcal{A}_t (\varphi ) = \lim_{ \varepsilon \to 0 } \mathcal{A}^\varepsilon_{ 0, t } (\varphi) ,
\end{equation*}
in $L^2 $ for every $t \in [0,T]$ and $\varphi \in \mathcal{S}(\mathbb{R})$.  
\end{proposition}

From the last proposition, thinking that the weak form of the singular term $\nabla u^2 $ is given by the last quantity, we can define a solution of \eqref{eq:SBE_general} as follows.

\begin{definition}
\label{def:energysol}
We say that an $\mathcal{S}^\prime(\mathbb{R})$-valued process $u=\{u (t, \cdot) : t\in [0,T] \}$ is a stationary energy solution of the SBE \eqref{eq:SBE_general} if all of the followings are satisfied.
\begin{enumerate}
\item The process $u$ satisfies the conditions \textbf{(S)} and \textbf{(EC)}. 
\item For all $\varphi \in \mathcal{S} (\mathbb{R} )$, the process 
\begin{equation*}
u_t(\varphi) - u_0 (\varphi) - D\int_0^t u_s (\Delta \varphi ) ds 
+ \Lambda \mathcal{A}_t (\varphi) ,
\end{equation*}
is a martingale with quadratic variation $2\chi D \| \nabla\varphi \|^2_{ L^2 (\mathbb{R} ) } t $ where $\mathcal{A}_\cdot$ is the process obtained in Proposition \ref{prop:nonlinear}. 
\item For all $\varphi \in \mathcal{S} (\mathbb{R} )$, writing $\hat{u}_t = u_{T-t}$ and $\hat{ \mathcal{A} }_t = - (\mathcal{A}_T - \mathcal{A}_{ T- t })$, the process
\begin{equation*}
\hat{u}_t (\varphi) - \hat{u}_0(\varphi) - D \int_0^t \hat{u}_s (\Delta\varphi) ds 
+ \Lambda \hat{\mathcal{A}}_t (\varphi) ,
\end{equation*}
is a martingale with quadratic variation $2\chi D \| \nabla\varphi \|^2_{L^2(\mathbb{R})}t$. 
\end{enumerate}
\end{definition}

Then there exists a unique-in-law stationary energy solution of \eqref{eq:SBE_general}. 
Existence was shown in \cite{gonccalves2014nonlinear} and then uniqueness was proved in \cite{gubinelli2018energy}.

Now, we are in a position to state our second theorem on scaling limits of the PEP. 
Here and in what follows, we assume the gradient condition \eqref{eq:gradient_condition}.
First, to derive the hydrodynamic limit, let $\pi^n_t$ be the empirical measure defined by 
\begin{equation*}
\pi^n_t
= \frac{1}{n}\sum_{x\in\mathbb Z}
\eta_x(t) \delta_{x/n}
\end{equation*}
where $\delta_{x/n}(\cdot)$ is the Dirac delta measure on $\mathbb R$ with the mass on $x/n$.  
When we impose the weak asymmetry in the sense that 
\begin{equation*}
p_n-q_n = \alpha_0/n 
\end{equation*}
for some $\alpha_0\in\mathbb R$, a drift term appears in the limiting equation. 
In this case, at least heuristically, the hydrodynamic equation of PEP would be 
\begin{equation}
\label{eq:pep_hydrodynamic_equation}
\partial_t\hat\rho_t={\Delta \tilde{D}(\hat\rho_t)+ \nabla\tilde{v}(\hat\rho_t)  =\nabla(D(\hat\rho_t)\nabla \hat\rho_t)}
+\nabla\tilde{v}(\hat\rho_t)  
\end{equation}
where 
\[
\tilde{D}(\rho)=\Phi_h(\rho)=\frac{1}{2}c(\kappa)\Phi_c(\rho),
\]
\begin{equation}
\label{eq:diffusion_coefficient}
D(\rho)=\frac{d}{d\rho} \tilde{D}(\rho),
\end{equation}
%the diffusion coefficient $D:[0,\kappa]\to\mathbb R_+$ is given by 
%\begin{equation}
%
%D(\rho) = \frac{1}{2}d(\kappa)\frac{\partial}{\partial\rho} E_{\nu_\rho}[c(\eta_0)]
%\end{equation}
and the macroscopic velocity $\tilde{v}$ is defined by 
\begin{equation*}
\tilde{v}(\rho)
= \alpha_0\Phi_r(\rho)=\alpha_0 \Phi_c(\rho)(c(\kappa)-\Phi_c(\rho)).
\end{equation*}

More precisely, $\pi^n_t$ converges in probability to a deterministic measure which is absolutely continuous with respect to the Lebesgue measure and its density $\hat\rho_t$ is the weak solution of \eqref{eq:pep_hydrodynamic_equation}. 
Although an exact result rigorously proving the hydrodynamic limit is missing, we believe that the procedure would be justified in a similar way as conventional works, see \cite{kipnis1999scaling}, particularly Chapter 5 for gradient systems, and Chapter 7 for non-gradient PEP with some specific jump rate.

%\begin{remark}
%say something about example. relation to zrp, interpolation. 
%\end{remark}

Next, we state the result for the fluctuating hydrodynamics. 
In what follows, we impose the \textit{weak asymmetry} in the sense that the absolute value of the discrepancy between $p_n$ and $q_n$ goes to zero more slowly than the scale $1/n$ for the space variable.   
%\begin{equation}
%\label{eq:weak_asymmetry}
%p_n - q_n = \alpha_n. 
%\end{equation}
In order to derive the SBE, we will take $p_n-q_n=\alpha n^{-1/2}$ for some $\alpha\in\mathbb R$. 
Here recall that the space $D([0,T],\mathcal S'(\mathbb R))$ is
the set of c\`adl\`ag (right-continuous and with left limits) trajectories in $\mathcal S'(\mathbb R)$.
Let $\mathcal{X}^n_\cdot \in D([0,T],\mathcal{S}'(\mathbb R))$ be the fluctuation field, associated to our PEP, defined on $\varphi \in \mathcal{S}(\mathbb R)$ by 
\begin{equation*}
\mathcal{X}^n_t(\varphi)
= \frac{1}{\sqrt{n}} \sum_{x\in\mathbb Z}
\overline\eta_x(t) T^-_{v_nt} \varphi^n_x
\end{equation*}
where we used the short-hand notation
\begin{equation*}
T^-_{v_nt} \varphi^n_x 
= \varphi\Big( {\frac{x-v_nt}{n}}\Big) 
\end{equation*}
and
\begin{equation}
\label{eq:pep_moving_frame}
v_n = \alpha n^{3/2}  \frac{\partial}{\partial\rho}\Phi_r(\rho)
\end{equation}
is the velocity of the moving frame. 
Above, bar over variables means centering with respect to $\nu_\rho$: $\overline{\eta}_x = \eta_x - E_{\nu_\rho}[\eta_x]$, for instance. 
Then we can show, under the weakly asymmetry $p_n-q_n=O(n^{-1/2})$, that the limit of the fluctuation fields $\mathcal{X}^n_\cdot$ is characterized by the stationary energy solution of the SBE.

\begin{theorem}
\label{thm:sbe_from_pep}
Fix $T>0$ and $\rho\in(0,\kappa)$. 
Assume $p_n-q_n=\alpha n^{-1/2}$ for some $\alpha\in\mathbb R$.  
Moreover, assume \eqref{eq:gradient_condition}, meaning that the process is of gradient type. 
Then, as $n\to \infty$, the sequence $\mathcal{X}^n_\cdot$ converges in distribution on $D([0,T],\mathcal{S}'(\mathbb R))$ to the stationary energy solution of the SBE
\begin{equation*}
\partial_tu 
= D(\rho) \Delta u 
- \Lambda(\rho) \nabla u^2
+ \sqrt{ 2\chi(\rho) D(\rho)} \nabla \dot W.
\end{equation*}
Above, $D(\rho)$ is defined in \eqref{eq:diffusion_coefficient}, which is strictly positive when $\rho\in(0,\kappa)$, and 
\begin{equation}
\label{eq:sbe_nonlinear_coeff}
\Lambda(\rho) = 
\frac{\alpha}{2} \frac{\partial^2}{\partial \rho^2} 
\Phi_{r}(\rho) 
=\frac{\alpha}{2}\frac{\partial^2}{\partial \rho^2} \big(\Phi_c(\rho)(c(\kappa)-\Phi_c(\rho))\big)
\end{equation}
and $\dot W$ is the one-dimensional space-time white-noise. 
\end{theorem}

\begin{remark}
In particular, when the asymmetry is weaker than the critical case, namely when $\lim_{n\to\infty}\sqrt{n}(p_n-q_n)=0$, the limiting equation is given by the stochastic heat equation with additive noise, so that we can cover the linear fluctuation result, see \cite[Theorem 2.1]{gonccalves2012crossover} for instance.
\end{remark}

\begin{remark}
In Theorem \ref{thm:sbe_from_pep}, we assumed that the dynamics of our PEP is initialized exactly from the product invariant measure.
This condition, however, could be made milder as in the case of the zero-range process~\cite{gonccalves2015stochastic}.
Indeed, if the initial measure is assumed to be sufficiently close to the product-invariant one with some proper order, then the estimate propagates in the whole duration, and, roughly speaking, we can interpret that the process is stationary as if it is initialized from the invariant measure. 
This makes it possible to show the derivation result in a similar way, combined with the assumption on the initial measure. 
\end{remark}

\subsection{Relationship to the simple exclusion process and the zero-range process}
The decomposable gradient PEPs share many nice properties with the simple exclusion process (SEP) and the zero-range process (ZR-process), and so we expect most of the results shown for these two processes would be extended to decomposable gradient PEPs, such as scaling limits for symmetric/asymmetric/finite-range/long-range versions, large deviations and so on. The similarities and differences between these processes are summarized below.   

First, these three processes have product invariant measures for both symmetric and asymmetric versions, and satisfy the gradient condition. In particular, they satisfy a stronger gradient condition than the usual gradient condition since the symmetric current $W^S$ is given as 
\[
W^S_{x,y}=h(\eta_x)-h(\eta_y)
\]
for some function $h$ which depends only on $\eta_x$. This feature plays a particularly important role in extending the lattice of the underlying space to crystal lattices, percolations, fractals, etc.

For the hydrodynamic limits and equilibrium fluctuations for weakly asymmetric processes on $\mathbb Z$, they share the form of the equations. The hydrodynamic limits are
\begin{equation*}
\partial_t\hat\rho_t=\Delta \tilde{D}(\hat\rho_t)+ \nabla\tilde{v}(\hat\rho_t) 
\end{equation*}
for the asymmetry $p_n-q_n=n^{-1}$ and the equilibrium fluctuations are
\begin{equation*}
\partial_tu 
= D(\rho) \Delta u 
- \Lambda(\rho) \nabla u^2
+ \sqrt{ 2\chi(\rho) D(\rho)} \nabla \dot W
\end{equation*}
for the asymmetry $p_n-q_n=n^{-1/2}$ 
where 
\[
D(\rho)=\frac{d}{d\rho} \tilde{D}(\rho), \quad \Lambda(\rho) =\frac{1}{2}\frac{d^2}{d\rho^2}\tilde{v}(\rho), \quad \chi(\rho)=\mathrm{Var}_{\nu_{\rho}}[\eta_0]
\]
and the {second-order} Einstein relation {(introduced in \cite{gonccalves2015einstein})}
\[
2\chi(\rho)D(\rho)=\tilde{v}(\rho)
\]
holds, which is shown in Lemma \ref{lem:Einstein} for PEPs. 
(Note that in \cite{gonccalves2015einstein}, the authors use the normalization $p_n-q_n=2an^{-1/2}$ and so some constants appear to be different from those in this paper.)
%noting that $\alpha=2$ in their notation.  
 Due to these relations, the pair of functions for the density $(\tilde{D}(\rho), \tilde{v}(\rho))$ determines these limiting equations. 

The explicit form of $\tilde{D}(\rho)$ and $\tilde{v}(\rho)$ differs according to the processes.
For the SEP, 
\[
\tilde{D}(\rho)=\frac{1}{2}\rho, \quad \tilde{v}(\rho)=\rho(1-\rho)
\]
and for the ZR-process,
\[
\tilde{D}(\rho)=\frac{1}{2}\Phi_g(\rho), \quad \tilde{v}(\rho)=\Phi_g(\rho)
\]
for some $\Phi_g(\rho)$ which depends on the jump rate $g :\{0,1,\dots \} \to \mathbb R_{\ge 0}$. In particular, when $g(m)=m$, namely the process corresponds to independent random walks, we have
\[
\tilde{D}(\rho)=\frac{1}{2}\rho, \quad \tilde{v}(\rho)=\rho.
\]
Finally, for our decomposable gradient PEPs,
\[
\tilde{D}(\rho)=\frac{1}{2}c(\kappa)\Phi_c(\rho), \quad \tilde{v}(\rho)=\Phi_c(\rho)(c(\kappa)-\Phi_c(\rho))
\]
for some $\Phi_c(\rho)$ which depends on the function $c :\{0,1,\dots \kappa \} \to \mathbb R_{\ge 0}$. In particular, when $c(m)=m$, namely the process is SEP$(\kappa)$,
\[
\tilde{D}(\rho)=\frac{\kappa}{2}\rho, \quad \tilde{v}(\rho)=\rho(\kappa-\rho).
\]

As the end of this remark, we discuss how to interpolate the SEP and ZR-process with our decomposable gradient PEPs. Starting with the simplest case, to interpolate the SEP and the independent random walks, we consider a normalized SEP$(\kappa)$ given by the function $c(m)=\frac{m}{\sqrt{\kappa}}$. For this normalized SEP$(\kappa)$, we have
\[
\tilde{D}(\rho)=\frac{1}{2}\rho, \quad \tilde{v}(\rho)=\rho\left(1-\frac{\rho}{\kappa}\right).
\]
Hence, when $\kappa=1$, they are consistent with the SEP, and as $\kappa \to \infty$, they converge to the case of independent random walks. More generally, for a strictly increasing function $g :\{0,1,\dots \} \to \mathbb R_{\ge 0}$ with $g(0)=0$, we consider the PEP given by the function $c(m)=g(m)/\sqrt{g(\kappa)}$ for each $\kappa$. 
For this PEP, we have
\[
\tilde{D}(\rho)=\frac{1}{2}\Phi_{g,\kappa}(\rho), \quad \tilde{v}(\rho)=\Phi_{g,\kappa}(\rho)\left(1-\frac{\Phi_{g,\kappa}(\rho)}{g(\kappa)}\right).
\]
Here $\Phi_{g,\kappa}(\rho)$ is the expectation of $g$ under the measure $\nu_{\rho}$ constructed with $c=g|_{\{0,1,\dots, \kappa\}}$.
Hence, if $g(\kappa) \to \infty$ and $\Phi_{g,\kappa}(\rho) \to \Phi_g(\rho)$ as $\kappa \to \infty$, then these functions converge to their counterparts of the ZR-process. Although it is an interesting question, we will not pursue here for which function $g$ this convergence really holds other than $g(m)=m$. It would also be interesting to know whether convergence in some stronger sense can be established.

\section{Characterization of the Gradient Condition}
\label{sec:pep_gradient}
In this section, we give a proof of Theorem \ref{thm:gradient_condition}. 
First, we show that the gradient condition for our process is characterized by the condition \eqref{eq:gradient_condition}. 

\begin{proposition}
The process $\{\eta(t):t\in[0,T]\}$ whose jump rate is given by \eqref{eq:jump_rate_assumption} is of gradient type if, and only if, the condition \eqref{eq:gradient_condition} is satisfied. 
In this case, particularly, the symmetric part of the current can be written as $W^S_{x,x+1}=h(\eta_x)-h(\eta_y)$ with $h(\eta_x)=\frac{1}{2}c(\kappa)c(\eta_x)$. 
\end{proposition}
\begin{proof}
For our process, note that the symmetric part of the current of particles between $x$ and $x+1$ satisfies 
\begin{equation*}
2W^S_{x,y}
= r(\eta_x,\eta_y) - r(\eta_y,\eta_x)
= c(\eta_x)d(\kappa-\eta_y)
- c(\eta_y) d(\kappa-\eta_x) 
\end{equation*}
where we set $y=x+1$. 
Hence, if the condition \eqref{eq:gradient_condition} holds, then 
\begin{equation*}
2W^S_{x,y}
= c(\eta_x) \big( c(\kappa)-c(\eta_y)\big)  
- c(\eta_y) \big( c(\kappa)-c(\eta_x)\big) 
=c(\eta_x)c(\kappa)-c(\eta_y)c(\kappa),
\end{equation*}
and thus the process is of gradient type. 
On the other hand, if the process is of gradient type, then there exists some function $h:\{0,\ldots,\kappa\}\to \mathbb R$ such that 
\begin{equation*}
W^S_{x,y}=h(\eta_x)-h(\eta_y)
\end{equation*}
by definition. 
%since $W^S_{x,y}$ is a function of $\eta_x$ and $\eta_y$. 
Moreover, note that any constant shift of $h$ does not change the current $W^S_{x,y}$. 
Therefore, without loss of generality we may assume $h(0)=0$. 
Then, setting $\eta_y=0$, we have 
\begin{equation*}
2W^S_{x,y} \big|_{\eta_y=0}{=r(\eta_x,0)-r(0,\eta_x)}
=c(\eta_x)d(\kappa)
=2h(\eta_x) 
\end{equation*}
by the fact that $c(0)=0$ and $h(0)=0$.
Consequently, 
\begin{equation*}
2W^S_{x,y}
= c(\eta_x)d(\kappa-\eta_y)
- c(\eta_y) d(\kappa-\eta_x)
= c(\eta_x)d(\kappa)
- c(\eta_y)d(\kappa) 
\end{equation*}
for any $\eta$.  
Setting $\eta_x =m$ and $\eta_y = \kappa$, we obtain 
\begin{equation*}
-c(\kappa)d(\kappa-m) = 
c(m)d(\kappa) - c(\kappa) d(\kappa). 
\end{equation*}
From the normalizing condition \eqref{eq:jump_rate_normalization}, the desired condition \eqref{eq:gradient_condition} is derived and this completes the proof. 
\end{proof}

\begin{proof}[Proof of Theorem \ref{thm:gradient_condition}]
By \cite[Proposition I.2.13]{liggett1985interacting}, the PEP is invariant under $\nu_{\rho}$ if, and only if,  
\[
E_{\nu_{\rho}}[L_nf]=0
\]
for any local function $f : \mathscr X \to \mathbb R$. Since we know that $E_{\nu_{\rho}}[S_nf]=0$ for any local function $f$, the condition in the last display is equivalent to 
\[
E_{\nu_{\rho}}[A_nf]
=\frac{n^2(p_n-q_n)}{2} E_{\nu_{\rho}}\bigg[\sum_{x\in\mathbb Z} 
\big( r(\eta_x,\eta_{x+1}) \nabla_{x,x+1}f
- r(\eta_{x+1},\eta_{x}) \nabla_{x+1,x}f \big)\bigg]=0
\]
for any local function $f$. Noting that $E_{\nu_{\rho}}[S_nf]=0$ and the relation \eqref{eq:reversible}, we have
\begin{align*}
& E_{\nu_{\rho}}\bigg[\sum_{x\in\mathbb Z} 
\big( r(\eta_x,\eta_{x+1}) \nabla_{x,x+1}f
- r(\eta_{x+1},\eta_{x}) \nabla_{x+1,x}f \big)\bigg]\\
&\quad= 2E_{\nu_{\rho}}\bigg[\sum_{x\in\mathbb Z} 
 r(\eta_x,\eta_{x+1}) \nabla_{x,x+1}f\bigg] \\
&\quad= -2 E_{\nu_{\rho}}\bigg[ \sum_{x\in\mathbb Z}\big( r(\eta_x,\eta_{x+1})-r(\eta_{x+1},\eta_x)\big)f\bigg] 
= -4 E_{\nu_{\rho}}\bigg[\sum_{x\in\mathbb Z}W^S_{x,x+1}f\bigg].
\end{align*}
%Hence, the PEP is invariant under %$\nu_{\rho}$ if, and only if, 
%\[
%\sum_{x\in\mathbb Z}W^S_{x,x+1} =0
%\]
%since the inner product with any local function $f$ is $0$. 
Hence, the PEP is invariant under $\nu_{\rho}$ if the PEP is gradient. 
Now, to prove the opposite direction, we apply the method used in the paper \cite{sasada2018green}. Let $(\phi_{i})_{i=0}^{\kappa}$ be an orthonormal basis of $L^2(\mu_{\rho})$ where $\mu_{\rho}$ is the one site marginal of $\nu_{\rho}$ with $\phi_{0}\equiv 1$. In other words, $\phi_{i} : \{0,1,\dots,\kappa\} \to \mathbb{R}$ are functions satisfying $E_{\nu_{\rho}}[\phi_{i}(\eta_x)\phi_{j}(\eta_x)]=\delta_{ij}$. Then, $(\phi_{ij})_{i,j=0}^{\kappa}$ forms an orthonormal basis of $L^2(\mu_{\rho}^2)$ where $\phi_{ij}$ is defined by $\phi_{ij}(\eta_1,\eta_2)=\phi_{i}(\eta_1)\phi_j(\eta_2)$ and $\mu_{\rho}^2$ is the two sites marginal of $\nu_{\rho}$. In particular, since $W^S_{0,1}$ is a mean zero function depending only on $\eta_0$ and $\eta_1$, we have that 
\begin{equation*}
W^S_{0,1}=\sum_{i=1}^{\kappa}a^{(0)}_{i} \phi_i(\eta_0)+\sum_{i=1}^{\kappa}a^{(1)}_{i} \phi_i(\eta_1) + \sum_{i,j=1}^{\kappa}b_{ij} \phi_{ij}(\eta_0,\eta_1)
\end{equation*}
for some coefficients $a^{(0)}_{i}, a^{(1)}_{i}$ and $b_{ij}$. Then, 
\[
E_{\nu_{\rho}}\bigg[\sum_{x\in\mathbb Z}W^S_{x,x+1} \phi_{ij}(\eta_0,\eta_1)\bigg]=0
\]
implies $b_{ij}=0$ for $1 \le i,j \le \kappa$ and
\[
E_{\nu_{\rho}}\bigg[\sum_{x\in\mathbb Z}W^S_{x,x+1} \phi_{i}(\eta_0)\bigg]=0
\]
implies $a^{(0)}_i+a^{(1)}_i=0$ for $1 \le i \le \kappa$.
This leads to
\begin{equation*}
W^S_{0,1}=\sum_{i=1}^{\kappa}a^{(0)}_{i} \left(\phi_i(\eta_0)- \phi_i(\eta_1)\right).
\end{equation*}
Hence, the PEP is of gradient type.
\end{proof}

Moreover, we have the following thermodynamic relation, which is referred to as the second-order Einstein relation as in the paper \cite{gonccalves2015einstein}, when the gradient condition \eqref{eq:gradient_condition} is satisfied. 

\begin{lemma}%[Green-Kubo formula]
\label{lem:Einstein}
Assume the gradient condition \eqref{eq:gradient_condition}. 
Recall that the diffusion coefficient $D(\rho)$ is defined by \eqref{eq:diffusion_coefficient}. 
Then 
\begin{equation}
\label{eq:gk_formula}
D(\rho)
= \frac{1}{2\chi(\rho)}\Phi_r(\rho). 
\end{equation}
In particular, $D(\rho)>0$ for any $\rho \in(0,\kappa)$. 
\end{lemma}
%\ms{(I suggest the new version of the proof.)}
\begin{proof}%[\ms{new version}]
First, note that by definition of the measure $\nu_{\rho}$, we have 
\begin{equation*}
E_{\nu_\rho}[c(\eta_0)]
= \lambda(\rho) E_{\nu_\rho}[d(\kappa-\eta_0)]. 
\end{equation*}
Combined with the gradient condition \eqref{eq:gradient_condition}, we have that 
\begin{equation*}
\Phi_c(\rho)=\lambda(\rho)(c(\kappa)-\Phi_c(\rho)),
\end{equation*}
which leads 
\begin{equation}
\label{eq:Phi_c}
\Phi_c(\rho)
=c(\kappa)\frac{\lambda(\rho)}{1+\lambda(\rho)}.  
\end{equation}
As a result, since $r(\eta_0,\eta_1)=c(\eta_0)(c(\kappa)-c(\eta_1))$ under the gradient condition,
\begin{equation*}
%\label{eq:Phi_r}
\Phi_r(\rho)=\Phi_c(\rho)(c(\kappa)-\Phi_c(\rho))
= c(\kappa)^2\frac{\lambda(\rho) }{(1+ \lambda(\rho))^2 }. 
\end{equation*}
On the other hand, 
\begin{equation*}
D(\rho) =\frac{1}{2} c(\kappa) \frac{d}{d\rho}\Phi_c(\rho) = \frac{1}{2} c(\kappa)^2 \frac{d}{d\rho}\frac{\lambda(\rho)}{1+\lambda(\rho)} 
= \frac{1}{2} 
 \frac{c(\kappa)^2}{(1+\lambda(\rho))^2}  \frac{d\lambda(\rho)}{d\rho}
\end{equation*}
where in the second identity, we used the relation \eqref{eq:Phi_c}.
Finally, recalling \eqref{eq:compressibility}, we have 
\begin{equation*}
2\chi(\rho)D(\rho)= \Phi_r(\rho)
\end{equation*}
and the assertion follows. 
\end{proof}

\section{The Boltzmann-Gibbs Principle}
\label{sec:pep_estimate}
\subsection{Preliminaries}
In this section, we summarize some basic estimates which are crucial to show Theorem \ref{thm:sbe_from_pep}.
First, let us recall some basic estimates which are concerned with the so-called $\mathscr H^{1,n}$ and $\mathscr H^{-1,n}$-norm, see \cite[Chapter 2]{komorowski2012fluctuations} for some properties of these norms.  
For each local $L^2(\nu_\rho)$ function $f$, let us define $\| f \|_{1,n}$ by 
\begin{equation*}
\begin{aligned}
\| f\|^2_{1,n}
&= \langle f, -L_n f\rangle_{L^2(\nu_\rho)}
= \langle f, -S_n f\rangle_{L^2(\nu_\rho)} \\
&= \frac{n^2}{2} \sum_{x\in\mathbb Z} 
E_{\nu_\rho}\big[ r(\eta_x,\eta_{x+1}) \big(\nabla_{x,x+1}f(\eta))^2 \big] . 
\end{aligned}
\end{equation*}
Moreover, for $f \in L^2(\nu_{\rho})$, we define the norm $\|\cdot \|_{-1,n}$ by the following variational formula. 
\begin{equation*}
\|f\|^2_{-1,n}
= \sup_{g} \big\{ 2\langle f, g \rangle_{L^2(\nu_\rho)}
- \|g\|^2_{1,n} \big\} 
\end{equation*}
where the supremum is taken over all local 
functions $g$.
Additionally, here and in what follows, let us introduce a local average of $\eta_x$ in a box with the size $\ell \in \mathbb N$ as follows: 
\begin{equation}
\label{eq:definition_local_average}
\overrightarrow \eta^\ell_x 
= \frac{1}{\ell} \sum_{y=0}^{\ell-1} \eta_{x+y}.  
\end{equation}
First, note that we have the following estimate, see \cite[Lemma 2.4]{komorowski2012fluctuations} (where the setting accounts for general state spaces, but functions that are not time-dependent) and  \cite[Lemma 4.3]{chang2001equilibrium}  (where time-dependent functions are considered though the process has finite state-space) for a proof.

%\begin{proposition}[Kipnis-Varadhan inequality]
%\label{prop:KV}
%There exists a universal constant $C>0$ such that for any smooth mean-zero function $f:[0,T]\to L^2(\nu_\rho)$, 
%\begin{equation*}
%\mathbb E_n \bigg[ \sup_{0\le t\le T} \bigg|\int_0^t f(s,\eta(s)) ds \bigg|^2 \bigg]
%\le C \int_0^T \| f(t,\cdot)\|^2_{-1,n}dt .
%\end{equation*}
%Above, $\eta$ is the process generated by $L_n$ and we omit the %dependency on $n$ by abuse of notation. 
%\end{proposition}

\begin{proposition}[Kipnis-Varadhan inequality]
\label{prop:KV}
There exists a universal constant $C>0$ such that for any local function $f$ satisfying $E_{\nu_{\sigma}}[f]=0$ for any $\sigma \in (0,\kappa)$ and $\varphi\in\mathcal{S}(\mathbb R)$, 
\begin{equation*}
\mathbb E_n \bigg[ \sup_{0\le t\le T} \bigg|\int_0^t F_n(s,\eta(s)) ds \bigg|^2 \bigg]
\le C \int_0^T \| F_n(t,\cdot)\|^2_{-1,n}dt
\end{equation*}
where $F_n(s,\eta(s))=\sum_{x \in \mathbb{Z}}\tau_x f(\eta(s))T^-_{v_n s}\varphi(x/n)$.
Here, $\eta$ is the process generated by $L_n$ and we omit the dependency on $n$ by abuse of notation. 
\end{proposition}

Then, we state the spectral gap estimate for the partial exclusion processes, see \cite[Section 3]{caputo2004spectral} for the proof.
%\kh{cite [sasada, Kac walk], also [caputo]?}
% \cite[Theorem 3.1]{caputo2004spectral} gives a uniform estimate, but here we need a local version as above. This is not proved in the paper, but it just says that the proof follows from the Kipnis Landim book. 

%\begin{definition}
%We say that a spectral gap estimate with intensity $\gamma>0$ holds if when there exists a uniform and finite constant $\gamma$ such that 
%\begin{equation*}
%\| f \|^2_{L^2(\nu_\rho)} \le \gamma \| f\|^2_{1}
%\le \frac{\gamma}{n^2} \| f\|^2_{1,n}\end{equation*}
%for any local function $f:\mathscr X \to \mathbb R$ such that $E_{\nu_\sigma}[f]=0$ for any $\sigma\in[0,\kappa]$. 
%\end{definition}

\begin{proposition}[Spectral gap estimate]
Let $f:\mathscr X\to \mathbb R$ be a local function satisfying $E_{\nu_\sigma}[f]=0$ for any $\sigma \in (0,\kappa)$ 
and $\mathrm{supp}(f)\subset \{ 0,\ldots,\ell_f-1\}$ for some $\ell_f\in\mathbb N$. 
There exists a universal and finite constant $\gamma>0$ such that  \begin{equation}
\label{eq:spectral_gap_estimate}
\| f\|^2_{L^2(\nu_\rho)} 
\le \gamma \ell_f^2 \sum_{x=0}^{\ell_f-2} 
E_{\nu_\rho}\big[ (\nabla_{x,x+1}f)^2\big]. 
\end{equation}
\end{proposition}

Recall the definition of the $\| \cdot\|_{1,n}$-norm of the local function $f$.  By the ellipticity condition \eqref{eq:jump_rate_ellipticity}, we have the bound
\begin{equation}
\label{eq:estimate_dirichlet_form}
\sum_{x=0}^{\ell_f-2} E_{\nu_\rho}[(\nabla_{x,x+1}f)^2]
\le \frac{1}{\varepsilon_0} \sum_{x=0}^{\ell_f-2} E_{\nu_\rho}[r(\eta_x,\eta_{x+1})(\nabla_{x,x+1}f)^2]
\le \frac{2}{\varepsilon_0 n^2} \| f\|^2_{1,n}. 
\end{equation}
Therefore, the spectral gap estimate yields the bound 
\begin{equation*}
%\label{eq:spectral_gap_consequence}
\| f\|^2_{L^2(\nu_\rho)}
\le \frac{2\gamma\ell_f^2}{\varepsilon_0 n^2} \| f\|^2_{1,n}. 
\end{equation*}
Using this consequence of the spectral gap estimate, we have the following bound of the $\mathscr H^{-1,n}$-norm. 

\begin{proposition}
\label{prop:H-1_norm_estimate}
Let $f:\mathscr X \to \mathbb R$ be a local function such that $E_{\nu_\sigma}[f]=0$ for any $\sigma\in (0,\kappa)$ and $\mathrm{supp}(f)\subset \{ 0,\ldots,\ell_f-1\}$ for some $\ell_f\in\mathbb N$. 
Then, we have that 
\begin{equation*}
\|f\|^2_{-1,n} 
\le \frac{2\gamma\ell_f^2}{\varepsilon_0n^2} \| f\|^2_{L^2(\nu_\rho)}. 
\end{equation*}
\end{proposition}

The proof of the previous estimate is given in \cite[Proposition 6]{gonccalves2014nonlinear} for exclusion processes with generic jump rates. 
Here we give a proof of Proposition \ref{prop:H-1_norm_estimate}, since the Dirichlet form of the process is different from those of exclusion processes, albeit the proof is quite similar. 
(See also the proof in \cite[Lemma 4.3]{gonccalves2015stochastic}.)

\begin{proof}[Proof of Proposition \ref{prop:H-1_norm_estimate}]
Let $g:\mathscr X\to \mathbb R$ be an arbitrary local function. 
For any $\ell\in\mathbb N$, let $\mathcal{F}_\ell$ be the $\sigma$-algebra generated by $\eta_0,\cdots,\eta_{\ell-1}$, and define $g_{\ell}=E_{\nu_\rho}[g|\mathcal{F}_{\ell}]$ and $\overline{g}_{\ell} = g_\ell- E_{\nu_\rho}[g_\ell|\overrightarrow\eta^\ell_0]$. 
%\ms{(Define $E_{\nu_\rho}[g_\ell|\overrightarrow\eta^\ell_0]$.)} 
Above, $E_{\nu_\rho}[g_\ell|\overrightarrow\eta^\ell_0]$ denotes the conditional expectation of the function $g_\ell$ on the $\sigma$-algebra generated by $\overrightarrow\eta^\ell_0$. 
%i.e. $\mathcal M_\ell:=\sigma(\overrightarrow\eta^\ell_0)$.  
First, note here that we can show that $E_{\nu_\rho}[f|\overrightarrow\eta^{\ell_f}_0]=0$.
\if0
%%%%%%%%%%%%%%%%%%%%%%%%%%%%%%%%%%%%%%
Indeed, by definition of the conditional expectation, we have that 
\begin{equation*}
\begin{aligned}
E_{\nu_\rho}[f|\overrightarrow\eta^{\ell_f}_0]
&= \sum_{K=0,\ldots,\kappa\ell_f}
E_{\nu_\rho}[f,\{\overrightarrow\eta^{\ell_f}_0=K/\ell_f\}]
\mathbf{1}_{\overrightarrow\eta^{\ell_f}_0=K/\ell_f} \\
&= \sum_{K=0,\ldots,\kappa\ell_f}
\frac{(\ell_f\lambda(\rho))^K}{Z_{\lambda(\rho)}} E_{\overline\nu_1}[f] 
\mathbf{1}_{\overrightarrow\eta^{\ell_f}_0=K/\ell_f}
=0
\end{aligned}\end{equation*} \ms{(I could not understand the last two equals.)}
since $E_{\nu_\sigma}[f]=0$ for any $\sigma\in (0,\kappa)$. 
\ms{(I suggest a proof for $E_{\nu_\rho}[f|\sigma(\sum_{x=0}^{\ell_f-1}\eta_x)]=0$ as follows.) 
%%%%%%%%%%%%%%%%%%%%%%%%%%%%%%%%%%%%%%%%%
\fi
For each $K=0,1,\ldots, \kappa\ell_f$, let $\nu_{\ell_f,K}$ be the probability measure on 
$$\mathscr X_{\ell_f,K}
\coloneqq \Big\{\eta \in \{0,1,\ldots, \kappa\}^{\{0,1,\ldots, \ell_f-1\}} : \sum_{x=0}^{\ell_f-1}\eta_x =K \Big\}
$$ 
defined as the restriction of $\nu_{\rho}$ on the set $\mathscr X_{\ell_f,K}$. 
It is easy to see that the measure is also the restriction of $\nu_{\sigma}$ on the set $\mathscr X_{\ell_f,K}$ for any $\sigma \in (0,\kappa)$. In particular, 
\[
E_{\nu_\sigma}[f]=\sum_{K=0}^{ \kappa\ell_f} \nu_\sigma( \mathscr X_{\ell_f,K})E_{\nu_{\ell_f,K}}[f]=0
\] 
for any $\sigma \in (0,\kappa)$. Hence, 
\[
\sum_{K=0}^{ \kappa\ell_f} \overline{\nu}_{\lambda}( \mathscr X_{\ell_f,K})E_{\nu_{\ell_f,K}}[f]= \sum_{K=0}^{ \kappa\ell_f}\frac{\lambda^K}{Z_{\lambda}^{\ell_f}}A_{\ell_f,K}E_{\nu_{\ell_f,K}}[f]=0
\]
for any $\lambda>0$ where $A_{\ell_f,K}$ is a positive number depending on $\ell_f, K$, but not on $\lambda$. Therefore, the polynomial $\sum_{K=0}^{ \kappa\ell_f}\lambda^K A_{\ell_f,K}E_{\nu_{\ell_f,K}}[f]$ in $\lambda$ is identically $0$ and since $A_{\ell_f,K} >0$, we conclude $E_{\nu_{\ell_f,K}}[f]=0$ for any $K$. Noting 
\[
E_{\nu_\rho}[f|\sigma(\sum_{x=0}^{\ell_f-1}\eta_x)] =\sum_{K=0}^{ \kappa\ell_f}E_{\nu_{\ell_f,K}}[f] \mathbf{1}_{\sum_{x=0}^{\ell_f-1}\eta_x =K}
\] 
we have $E_{\nu_\rho}[f|\sigma(\sum_{x=0}^{\ell_f-1}\eta_x)] =0$.

Then, note that by definition the relation $\langle f,g\rangle_{\rho}
=\langle f, \overline g_{\ell_f} \rangle_\rho$ is deduced straightforwardly where we used the short-hand notation $\langle \cdot,\cdot\rangle_\rho = \langle \cdot, \cdot\rangle_{L^2(\nu_\rho)}$. 
Moreover, note that 
\begin{equation}
\label{eq:dirichlet_form_estimate_g_ell}
\sum_{x=0}^{\ell_f-2} 
E_{\nu_\rho}\big[ (\nabla_{x,x+1}\overline g_{\ell_f})^2\big]
\le \sum_{x=0}^{\ell_f-2} 
E_{\nu_\rho}\big[ (\nabla_{x,x+1}g)^2\big]
\le \frac{2}{\varepsilon_0 n^2} \| g\|^2_{1,n}
\end{equation}
where we used Jensen's inequality in the first estimate and
the second estimate follows from \eqref{eq:estimate_dirichlet_form}. 
Hence, by Young's inequality, we have that 
\begin{equation*}
\begin{aligned}
2\langle f,g \rangle_\rho
&=2\langle f,\overline g_{\ell_f} \rangle_\rho
\le A \langle f,f \rangle_\rho
+ \frac{1}{A} \langle \overline g_{\ell_f},\overline g_{\ell_f} \rangle_\rho \\
&\le A \langle f,f \rangle_\rho
+ \frac{\gamma\ell_f^2}{A} 
\sum_{x=0}^{\ell_f-2} 
E_{\nu_\rho}\big[ (\nabla_{x,x+1}\overline g_{\ell_f})^2\big] \\
&\le A \langle f,f \rangle_\rho
+ \frac{2\gamma \ell_f^2}{\varepsilon_0 n^2A} 
\| g \|^2_{1,n}
\end{aligned}
\end{equation*}
for any $A>0$ where we used Proposition \ref{eq:spectral_gap_estimate} in the penultimate estimate noting that the local function $\overline{g}_{\ell_f}$ satisfies the assumption of the spectral gap estimate, and the last estimate is a consequence of the estimate \eqref{eq:dirichlet_form_estimate_g_ell}. 
Therefore, by choosing $A=2\gamma \ell_f^2 \varepsilon_0^{-1}n^{-2}$, the assertion follows by definition of the norm $\|\cdot \|_{-1,n}$. 
%we have that 
%\begin{equation*}
%2 \langle f,g \rangle_\rho- \| g \|^2_{1,n}
%\le \frac{2\gamma \ell_f^2}{\varepsilon_0} \| f\|^2_{L^2(\nu_\rho)}. 
%\end{equation*}
\end{proof}

Additionally, we have the following result on the $\mathscr H^{-1,n}$-norm of functions with disjoint supports, see \cite[Proposition 7]{gonccalves2014nonlinear} for the proof. 

\begin{proposition}
%\label{prop:H-1_orthogonality}
Let $m\in\mathbb N$ be given.
Take a sequence $k_0< \cdots < k_m$ in $\mathbb Z$ and let $f_1,\ldots, f_m:\mathscr X \to \mathbb R$ be a sequence of local functions such that $\mathrm{supp}(f_i)\subset \{ k_{i-1},\ldots, k_i-1\}$.
Define $\ell_{f_i} = k_i- k_{i-1}$. 
Assume that $E_{\nu_\sigma}[f_i]=0$ for any $\sigma\in (0,\kappa)$.
Then, 
\begin{equation*}
\| f_1 + \cdots f_m \|^2_{-1,n}
\le \sum_{i=1}^m \frac{2\gamma\ell_{f_i}^2}{\varepsilon_0n^2} \| f_i\|^2_{L^2(\nu_\rho)}. 
\end{equation*}
\end{proposition} 

Note here that we can easily extend the estimate in Proposition \ref{prop:H-1_orthogonality} to a sequence of infinitely many functions as follows.

\begin{proposition}
\label{prop:H-1_orthogonality}
Let $f_i:\mathscr X \to \mathbb R, i=1,2,\ldots$ be a sequence of local functions such that $\mathrm{supp}(f_i) \subset \Lambda_i$ where $\Lambda_i=\{a_i, a_i+1,\dots,b_i\}$ for some $a_i < b_i \in \mathbb Z$ and $\Lambda_i \cap \Lambda_j = \emptyset$ for any $i \neq j$. 
Define $\ell_{i} =|b_i-a_i|+1$. 
Assume that $E_{\nu_\sigma}[f_i]=0$ for any $\sigma \in (0,\kappa)$ and $\sum_{i=1}^{\infty}E_{\nu_\rho}[f_i^2] <\infty$. 
Then, 
\begin{equation*}
\Big\| \sum_{i=1}^{\infty} f_i \Big\|^2_{-1,n}
\le \sum_{i=1}^{\infty} \frac{2\gamma\ell_{i}^2}{\varepsilon_0n^2} \| f_i\|^2_{L^2(\nu_\rho)}. 
\end{equation*}
\end{proposition} 

\begin{proof}
First note that the infinite series $\sum_{i=1}^{\infty} f_i $ is in $L^2(\nu_\rho)$ since $f_i(\eta)$ and $f_j(\eta)$ are independent under $\nu_{\rho}$, and so the $\mathscr H^{-1,n}$ is well-defined. By definition,
\begin{equation*}
\|\sum_{i=1}^{\infty} f_i \|^2_{-1,n}
= \sup_{g} \big\{ 2\langle \sum_{i=1}^{\infty} f_i , g \rangle_{L^2(\nu_\rho)}
- \|g\|^2_{1,n} \big\} = \sup_{g} \big\{ 2 \sum_{i=1}^{\infty} \langle f_i , g \rangle_{L^2(\nu_\rho)}
- \|g\|^2_{1,n} \big\}
\end{equation*}
where the supremum is taken over all local functions $g$.
Let $\mathcal{F}_i$ be the $\sigma$-algebra generated by $\eta_{a_i},\eta_{a_{i}+1},\ldots, \eta_{b_i}$. Then, 
\[
\langle f_i , g \rangle_{L^2(\nu_\rho)}= \langle f_i , g_i \rangle_{L^2(\nu_\rho)} =  \langle f_i , \tilde{g_i} \rangle_{L^2(\nu_\rho)}
\]
where $g_i:=E_{\nu_{\rho}}[g| \mathcal{F}_i]$ and $\tilde{g_i}:=g_i-E_{\nu_{\rho}}[g_i | \sigma( \sum_{x \in \Lambda_i} \eta_x)]$. Then, by the spectral gap estimate,
\[
E_{\nu_\rho}[ \tilde{g_i}^2] \le \gamma \ell_{i}^2 \sum_{x=a_i}^{b_i-1} E_{\nu_\rho}[(\nabla_{x,x+1} \tilde{g_i})^2] =  \gamma \ell_{i}^2 \sum_{x=a_i}^{b_i-1} E_{\nu_\rho}[(\nabla_{x,x+1} g_i)^2]. 
\]
Since $\nabla_{x,x+1}g_i= E_{\nu_\rho}[\nabla_{x,x+1}g | \mathcal{F}_i]$ when $x,x+1 \in \Lambda_i$, by Jensen's inequality 
\[
E_{\nu_\rho}[ \tilde{g_i}^2] \le   \gamma \ell_{i}^2 \sum_{x=a_i}^{b_i-1} E_{\nu_\rho}[(\nabla_{x,x+1} g)^2]
\]
holds. Hence, for any local function $g$,
\begin{align*}
\|g\|^2_{1,n}  & = \frac{n^2}{2} \sum_{x\in\mathbb Z} 
E_{\nu_\rho}\big[ r(\eta_x,\eta_{x+1}) \big(\nabla_{x,x+1}g)^2 \big]   \ge  \frac{n^2}{2} \sum_{i=1}^{\infty}\sum_{x=a_i}^{b_i-1} 
E_{\nu_\rho}\big[ r(\eta_x,\eta_{x+1}) \big(\nabla_{x,x+1}g)^2 \big]  \\
& \ge   \frac{n^2 \varepsilon_0}{2} \sum_{i=1}^{\infty}\sum_{x=a_i}^{b_i-1} 
E_{\nu_\rho}\big[ \big(\nabla_{x,x+1}g)^2 \big]  \ge \frac{n^2 \varepsilon_0}{2 \gamma} \sum_{i=1}^{\infty} \frac{1}{\ell_i^2}
E_{\nu_\rho}\big[  \tilde{g_i}^2 \big].
\end{align*}
For a given local function $g$, $\langle f_i , g \rangle_{L^2(\nu_\rho)} = 0$ except for finitely many $i$. Let $\mathcal{I}_g \subset \mathbb N$ be the set of $i$ such that $\langle f_i , g \rangle_{L^2(\nu_\rho)} \neq 0$.  Combining the above estimates,
\begin{align*}
 & 2 \sum_{i=1}^{\infty} \langle f_i , g \rangle_{L^2(\nu_\rho)}
- \|g\|^2_{1,n}  = 2 \sum_{i \in \mathcal{I}_g} \langle f_i , \tilde{g_i} \rangle_{L^2(\nu_\rho)} -  \|g\|^2_{1,n}  \\
& \le 2 \sum_{i \in \mathcal{I}_g} \langle f_i , \tilde{g_i} \rangle_{L^2(\nu_\rho)} - \frac{n^2 \varepsilon_0}{2 \gamma} \sum_{i=1}^{\infty} \frac{1}{\ell_i^2}
E_{\nu_\rho}\big[  \tilde{g_i}^2 \big] \le 2 \sum_{i \in \mathcal{I}_g} \langle f_i , \tilde{g_i} \rangle_{L^2(\nu_\rho)} - \frac{n^2 \varepsilon_0}{2 \gamma} \sum_{i \in  \mathcal{I}_g} \frac{1}{\ell_i^2}  E_{\nu_\rho}\big[  \tilde{g_i}^2 \big] \\
& =  \sum_{i \in \mathcal{I}_g} \big( 2  \langle f_i , \tilde{g_i} \rangle_{L^2(\nu_\rho)} -  \frac{n^2 \varepsilon_0}{2 \gamma \ell_i^2} 
E_{\nu_\rho}\big[  \tilde{g_i}^2 \big] \big) \le  \sum_{i \in \mathcal{I}_g} \frac{2 \gamma \ell_i^2}{n^2 \varepsilon_0} E_{\nu_\rho}\big[ f_i^2 \big] \le  \sum_{i=1}^{\infty} \frac{2 \gamma \ell_i^2}{n^2 \varepsilon_0}  \| f_i\|^2_{L^2(\nu_\rho)}. 
\end{align*}
\end{proof}

Finally, we state the so-called \textit{equivalence of ensembles}, which enables us to approximate a local function in terms of local averages.
Recall the definition of the local average \eqref{eq:definition_local_average}. 
Then, we have the following result, see \cite[Proposition 3.1]{gonccalves2013scaling} for the proof.

\begin{proposition}[Equivalence of Ensembles]
Let $f:\mathscr X\to \mathbb R$ be a local function {satisfying $\mathrm{supp}(f)\subset \{ 0,\ldots,\ell_f-1\}$ for some $\ell_f\in\mathbb N$.}
Then, there exists a constant $C=C(\rho,f)>0$ such that 
\if0
\begin{equation*}
\sup_{\eta\in\mathscr X} 
\bigg| E_{\nu_\rho}[f|\overrightarrow \eta^\ell_0]
- E_{\nu_{\overrightarrow \eta^\ell_0}}[f] 
- \frac{2}{\ell} \chi(\eta^\ell_0) 
\frac{d^2}{d\sigma^2} E_{\nu_\sigma}[f]\Big|_{\sigma=\overrightarrow\eta^\ell_0} 
%\varphi_f''(\eta^\ell_{x_0}) 
\bigg| 
\le C \ell^{-2}
\end{equation*}
\fi
%\ms{(How about to rewrite in the following way? I prefer this style also in the rest of the paper.) 
\begin{equation*}
\sup_{\eta\in\mathscr X} 
\bigg| E_{\nu_\rho}[f|\overrightarrow \eta^\ell_0]
- \Phi_f( \overrightarrow \eta^\ell_0)
- \frac{2}{\ell} \chi(\overrightarrow\eta^\ell_0) 
 \Phi_f''(\overrightarrow\eta^\ell_0) 
\bigg| 
\le C \ell^{-2}
\end{equation*}
where $$\Phi_f''(\rho)=\frac{d^2}{d\rho^2} E_{\nu_\rho}[f]$$
for any $\ell \ge \ell_f$. 
%\textcolor{magenta}{I also prefer this notation, and thanks for correcting the upper arrow in $\chi$.}
%where recall that $\chi(\rho)=\mathrm{Var}_{\nu_\rho}[\eta_0]$ denotes the static compressibility.  
\end{proposition}

Note that the equivalence of ensembles is a statement concerning measures $\{\nu_\rho\}_{\rho\in [0,\kappa]}$, which has nothing to do with the dynamics. 
The equivalence of ensembles is used to show the following result, see \cite[Proposition 3.2]{gonccalves2013scaling}. 

\begin{proposition}
Let $f:\mathscr X\to \mathbb R$ be a local function  {satisfying $\mathrm{supp}(f)\subset \{ 0,\ldots,\ell_f-1\}$ for some $\ell_f\in\mathbb N$.}
Then there exits a constant $C=C(f,\rho)$ such that 
\begin{equation*}
\begin{aligned}
& E_{\nu_\rho}\Big[ \Big\{
E_{\nu_\rho}[f|\overrightarrow \eta^\ell_0]
- E_{\nu_\rho}[f] 
- \frac{d}{d\rho}E_{\nu_\rho}[f]
\big( \overrightarrow\eta^\ell_0 - \rho \big) \\
&\quad- \frac{1}{2} \frac{d^2}{d\rho^2} E_{\nu_\rho}[f] 
\Big( \big(\overrightarrow\eta^\ell_0-\rho\big)^2 - \frac{\chi(\rho)}{\ell} \Big) \Big\}^2 \Big]
\le C \ell^{-3}
\end{aligned}
\end{equation*}
for any $\ell \ge \ell_f$. 
\end{proposition}

In particular, for a local function $f$ such that 
\begin{equation}
\label{eq:2bg_condition}
E_{\nu_\rho}[f] 
= %(\partial/\partial\rho) 
\frac{d}{d\rho} E_{\nu_\rho}[f]
= 0 ,
\end{equation}
we have that 
\begin{equation}
\label{eq:EE_consequence}
E_{\nu_\rho}\Big[ \Big\{
E_{\nu_\rho}[f|\overrightarrow \eta^\ell_0]
- \frac{1}{2} \frac{d^2}{d\rho^2} E_{\nu_\rho}[f] 
\Big( \big(\overrightarrow\eta^\ell_0-\rho\big)^2 - \frac{\chi(\rho)}{\ell} \Big) \Big\}^2 \Big]
\le C\ell^{-3} . 
\end{equation}

\subsection{The Boltzmann-Gibbs Principle}
Next we state the first-order Boltzmann-Gibbs principle and its second-order version that was introduced in \cite{gonccalves2014nonlinear}.
In what follows, we use the short-hand notation 
\begin{equation*}
\varphi^n_x(t) = T^-_{v_nt}\varphi(x/n) 
\end{equation*}
for each $\varphi\in\mathcal{S}(\mathbb R)$.
Although the following result may be extended to a more general class of time-dependent test functions with possibly rougher regularity, we may only consider the above form of test functions for the sake of simplicity.

\begin{proposition}[The first-order Boltzmann-Gibbs principle]
\label{prop:1bg}
For any local function $f:\mathscr X\to \mathbb R$ and for any $\varphi\in\mathcal{S}(\mathbb R)$, we have that 
\begin{equation*}
\lim_{n\to\infty}\mathbb E_n\bigg[\sup_{0\le t\le T} \bigg| \int_0^t 
\frac{1}{\sqrt{n}} \sum_{x\in\mathbb Z} 
\Big( \tau_xf(\eta(s)) - E_{\nu_\rho}[f]
- \frac{d}{d\rho} E_{\nu_\rho}[f] \overline\eta_x(s) \Big) \varphi^n_x(s) ds 
\bigg|^2 \bigg]
=0.
%\le C \frac{T^{3/2}}{n} \| \varphi\|^2_{L^2(\mathbb R)} .
\end{equation*}
\end{proposition}

Next, we state the second-order Boltzmann-Gibbs principle. 
Recall here the definition of the local average \eqref{eq:definition_local_average}. 

\begin{proposition}[The second-order Boltzmann-Gibbs principle]
\label{prop:2bg}
Let $f:\mathscr X\to \mathbb R$ be a local function satisfying \eqref{eq:2bg_condition}. 
Then there exists $C=C(\rho,f)>0$ such that for any $\ell \in\mathbb N$ and $\varphi\in\mathcal{S}(\mathbb R)$,  
\begin{equation*}
\begin{aligned}
&\mathbb E_n\bigg[ \sup_{0\le t\le T} \bigg| \int_0^t \sum_{x\in \mathbb Z}
\Big( \tau_x f(\eta(s)) 
- \frac{1}{2}\frac{d^2}{d\rho^2} E_{\nu_\rho}[f]
\big(\overrightarrow \eta^\ell_x(s)-\rho\big)^2 \Big) \varphi^n_x(s) ds \bigg|^2  \bigg] \\
&\quad \le C\Big( \frac{\ell}{n^2} + \frac{T}{\ell^2} \Big) 
n T\| \varphi\|^2_{L^2(\mathbb R)}. 
\end{aligned}
\end{equation*}
\end{proposition}

As pointed out in \cite[Remark 11]{gonccalves2014nonlinear}, the first-order Boltzmann-Gibbs is straightforward from the second-order one by taking $f-E_{\nu_\rho}[f] - (d/d\rho)E_{\nu_\rho}[f]\overline\eta_0$ as the local function in Proposition \ref{prop:2bg}.  

   {In the following, we prove Proposition \ref{prop:2bg} for a local function $f$ satisfying $\mathrm{supp}(f)\subset \{ 0,\ldots,\ell_f-1\}$ for some $\ell_f\in\mathbb N$. Once the result is established, for a general local function $f$, we can apply the result to a properly shifted version of $f$ and show that the error term converges to $0$, which implies Proposition \ref{prop:2bg} holds for any local function.}

%In what follows, we give a proof of the second-order Boltzmann-Gibbs principle. 
The proof boils down to showing the following two estimates. 

%\ms{(For the next two lemmas, we need to say something about the case when $\mathrm{supp}(f) \subset \{ 0,\ldots, \ell_0-1\}$ does not hold.)}

\begin{lemma}%[One-block estimate]
\label{lem:one-block}
Let $f:\mathscr X\to \mathbb R$ be a local function satisfying the condition \eqref{eq:2bg_condition} and  $\mathrm{supp}(f) \subset \{ 0,\ldots, \ell_0-1\}$ for some $\ell_0\in\mathbb N$.
Then, there exists $C=C(\rho,f)>0$ such that for any $\varphi\in\mathcal{S}(\mathbb R)$ and any $\ell\in\mathbb N$ such that $\ell \ge \ell_0$,   
\begin{equation*}
\mathbb E_n\bigg[\sup_{0\le t\le T} \bigg| 
\int_0^t \sum_{x\in\mathbb Z} 
\tau_x \big( f(\eta(s)) - E_{\nu_\rho}[f|\overrightarrow\eta^\ell_0(s)]\big) 
\varphi^n_x(s) ds \bigg|^2\bigg]
\le C \frac{\ell }{n} T\| \varphi\|^2_{L^2(\mathbb R)}. 
\end{equation*}
\end{lemma}

\begin{lemma}
\label{lem:2bg_key}
Let $f:\mathscr X\to \mathbb R$ be a local function satisfying the condition \eqref{eq:2bg_condition} and $\mathrm{supp}(f) \subset\{ 0,\ldots, \ell_0-1\}$ for some $\ell_0\in\mathbb N$.
Then, there exists $C=C(\rho,f)>0$ such that for any $\varphi\in\mathcal{S}(\mathbb R)$ and any $\ell\in\mathbb N$ such that $\ell\ge \ell_0$,   
\begin{equation*}
\begin{aligned}
&\mathbb E_n\bigg[\sup_{0\le t\le T} \bigg| 
\int_0^t \sum_{x\in\mathbb Z} 
\tau_x \Big( E_{\nu_\rho}[f|\overrightarrow\eta^\ell_0(s)]
- \frac{1}{2}\frac{d^2}{d\rho^2} 
E_{\nu_\rho}[f] \Big( (\overrightarrow\eta^\ell_0(s) -\rho)^2 - \frac{\chi(\rho)}{\ell} \Big) \Big)
\varphi^n_x(s) ds \bigg|^2\bigg] \\
&\quad\le C \frac{n T}{\ell^2} \| \varphi\|^2_{L^2(\mathbb R)}. 
\end{aligned}
\end{equation*}
\end{lemma}

Hereafter we give a proof of Lemma \ref{lem:one-block}.
For that purpose, first we show the following rough bound by the Kipnis-Varadhan inequality. 

\begin{lemma}
\label{lem:one-block_rough_bound}
Let $f:\mathscr X\to \mathbb R$ be a local function {satisfying} $\mathrm{supp}(f) \subset \{ 0,\ldots, \ell_0-1\}$ for some $\ell_0\in\mathbb N$.
Then, there exists $C=C(\rho,f)>0$ such that for any $\varphi\in\mathcal{S}(\mathbb R)$, 
\begin{equation*}
\begin{aligned}
\mathbb E_n\bigg[\sup_{0\le t\le T} \bigg| \int_0^t 
\sum_{x\in\mathbb Z} \tau_x \big( f(\eta(s)) - E_{\nu_\rho}[f|\overrightarrow \eta^{\ell_0}_0(s)] \big)
\varphi^n_x(s) ds \bigg|^2 \bigg]
\le C \frac{\ell_0^3}{n} 
T\| \varphi\|^2_{L^2(\mathbb R)}. 
\end{aligned}
\end{equation*}
\end{lemma}
\begin{proof}
The assertion easily follows by splitting the sum in $x$ into $\ell_0$ boxes. 
Indeed, by the Cauchy-Schwarz inequality, the left-hand side of the assertion is bounded {from} above by 
\begin{equation*}
\ell_0 \sum_{i=0}^{\ell_0-1} 
\mathbb E_n\bigg[ \sup_{0\le t\le T} \bigg| \int_0^t \sum_{z\in\mathbb Z} 
\tau_{z\ell_0 + i} \big( f(\eta(s)) -E_{\nu_\rho}[f|\overrightarrow\eta^{\ell_0}_0(s)]\big)
\varphi^n_{z\ell_0+i}(s)ds \bigg|^2 \bigg] 
\end{equation*}
multiplied by some constant $C(\rho,f)>0$.  
Since each function in the summand in the above display is mean-zero with respect to $\nu_\sigma$ for any $\sigma\in (0,\kappa)$, then by Propositions \ref{prop:KV} and \ref{prop:H-1_orthogonality}, we bound last display from above by a constant times
\begin{equation*}
\frac{\ell_0^3}{n^2} \sum_{i=0}^{\ell_0-1} \sum_{z\in\mathbb Z} 
\int_0^T 
(\varphi^n_{z\ell_0+i}(s))^2ds 
\end{equation*}
so that we
obtain the desired bound. 
\end{proof}

Next, we improve the bound of the one-block estimate by a multi-scale argument.

\begin{lemma}[Doubling the box]
\label{lem:doubling_box}
Let $f:\mathscr X\to\mathbb R$ be a local function satisfying the condition \eqref{eq:2bg_condition}{and $\mathrm{supp}(f) \subset\{ 0,\ldots, \ell_0-1\}$ for some $\ell_0\in\mathbb N$}. 
Then, there exists $C=C(\rho,f)>0$ such that for any $\ell \ge \ell_0$ and any $\varphi \in \mathcal{S}(\mathbb R)$,  
\begin{equation*}
\begin{aligned}
\mathbb E_n \bigg[\sup_{0\le t\le T} \bigg| \int_0^t \sum_{z\in\mathbb Z} 
\tau_z \big( E_{\nu_\rho}[f|\overrightarrow\eta^{\ell}_0(s)]
- E_{\nu_\rho}[f|\overrightarrow\eta^{2\ell}_0(s)] \big) \varphi^n_z(s) ds 
\bigg|^2 \bigg]
\le  \frac{C\ell }{n} T\| \varphi\|^2_{L^2(\mathbb R)}. 
\end{aligned}
\end{equation*}
\end{lemma}
\begin{proof}
The strategy of the proof is the same as that of the previous lemma, and the assertion easily follows from \eqref{eq:EE_consequence}.
Indeed, note that the support of $E_{\nu_\rho}[f|\overrightarrow\eta^{\ell}_0]
- E_{\nu_\rho}[f|\overrightarrow\eta^{2\ell}_0]$ is included in the interval $\{0, \ldots, 2\ell -1\} $. 
Then, again by splitting the sum in $x$ into boxes of size $2\ell$ and using Proposition \ref{prop:KV} we have that 
\begin{equation*}
\begin{aligned}
& 2\ell \sum_{i=0}^{2\ell-1} 
\mathbb E_n \bigg[\sup_{0\le t \le T} \bigg| \int_0^t 
\sum_{z\in\mathbb Z} \tau_{2\ell z + i} 
\big( E_{\nu_\rho}[f|\overrightarrow\eta^{\ell}_0(s)]
- E_{\nu_\rho}[f|\overrightarrow\eta^{2\ell}_0(s)]\big)
\varphi^n_{2\ell z+ i}(s) ds \bigg|^2 \bigg] \\
&\quad \leq 2\ell \sum_{i=0}^{2\ell-1} 
 \int_0^T \Big\|
\sum_{z\in\mathbb Z} \tau_{2\ell z + i} 
\big( E_{\nu_\rho}[f|\overrightarrow\eta^{\ell}_0(s)]
- E_{\nu_\rho}[f|\overrightarrow\eta^{2\ell}_0(s)]\big)
\varphi^n_{2\ell z+ i}(s) \Big\|_{-1,n}^2 ds.
\end{aligned}\end{equation*}
At this point, we can use Proposition  \ref{prop:H-1_orthogonality} to bound last display by 
\begin{equation*}
\begin{aligned}
&\frac{C\ell^3T}{n^2} \sum_{i=0}^{2\ell-1} 
\sum_{z\in\mathbb Z} E_{\nu_\rho} \big[ (E_{\nu_\rho}[f|\overrightarrow\eta^{\ell}_0]
- E_{\nu_\rho}[f|\overrightarrow\eta^{2\ell}_0])^2 \big]
\varphi^n_{2\ell z+ i}(s)^2 
\end{aligned}\end{equation*}
for some $C>0$.
The proof follows from the next estimate which is a consequence of \eqref{eq:EE_consequence}: 
\begin{equation*}
E_{\nu_\rho} \big[ (E_{\nu_\rho}[f|\overrightarrow\eta^{\ell}_0]
- E_{\nu_\rho}[f|\overrightarrow\eta^{2\ell}_0])^2 \big]
\le {C(\rho,f)} \ell^{-2}.
\end{equation*}
\end{proof}

%Now we give a proof of the one-block estimate with the proper order. 

\begin{proof}[Proof of Lemma \ref{lem:one-block}]
First, we show the assertion when $\ell = 2^M \ell_0$ for some $M\in\mathbb N$. 
Note that 
\begin{equation*}
\begin{aligned}
f(\eta) - E_{\nu_\rho}[f|\overrightarrow\eta^\ell_0]
= f(\eta) - E_{\nu_\rho}[f|\overrightarrow\eta^{\ell_0}_0]
+ \sum_{m=1}^M \Big\{ 
E_{\nu_\rho}[f|\overrightarrow\eta^{2^{m-1}\ell_0}_0]
- E_{\nu_\rho}[f|\overrightarrow\eta^{2^m\ell_0}_0]
\Big\} .
\end{aligned}
\end{equation*}
By Minkowski's inequality, we have the bound 
\begin{equation*}
\begin{aligned}
&\mathbb E_n \bigg[\sup_{0\le t\le T} \bigg| \int_0^t \sum_{x\in\mathbb Z} 
\sum_{m=1}^M 
\tau_x \big( E_{\nu_\rho}[f|\overrightarrow\eta^{2^{m-1}\ell_0}_0(s)]
- E_{\nu_\rho}[f|\overrightarrow\eta^{2^m\ell_0}_0(s)] \big) \varphi^n_x(s) ds 
\bigg|^2 \bigg]^{1/2} \\
&\quad\le C
\sum_{m=1}^M 
\mathbb E_n \bigg[\sup_{0\le t\le T} \bigg| \int_0^t \sum_{x\in\mathbb Z} 
\tau_x \big( E_{\nu_\rho}[f|\overrightarrow\eta^{2^{m-1}\ell_0}_0(s)]
- E_{\nu_\rho}[f|\overrightarrow\eta^{2^m\ell_0}_0(s)] \big) \varphi^n_x(s) ds 
\bigg|^2 \bigg]^{1/2} \\
&\quad\le C \sum_{m=1}^M \Big( \frac{2^{m-1}\ell_0 }{n}  \int_0^T\| \varphi(s)\|^2_{L^2(\mathbb R)}ds \Big)^{1/2} 
\le C 
\Big( \frac{2^M\ell_0}{n}\Big)^{1/2}  
%\frac{(2^M\ell_0)^{1/2} T^{1/2}}{\sqrt{n}} 
\Big(\int_0^T\| \varphi(s)\|^2_{L^2(\mathbb R)}ds \Big)^{1/2} 
\end{aligned}
\end{equation*}
for some $C=C(\rho,f)>0$ where we used Lemma \ref{lem:doubling_box} in the second estimate. 
Then the assertion immediately follows from the above estimate and Lemma \ref{lem:one-block_rough_bound}. 
Finally, for general $\ell$, it suffices to take a number $M\in\mathbb N$ in such a way that $2^M\ell_0< \ell < 2^{M+1} \ell_0$ and then directly estimate the quantity
\begin{equation*}
\int_0^t \sum_{x\in\mathbb Z} 
\tau_x \big( E_{\nu_\rho}[f| \overrightarrow\eta^{2^M\ell_0}(s)] 
- E_{\nu_\rho}[f|\overrightarrow\eta^\ell_0(s)] \big)
\varphi^n_x (s) ds,
\end{equation*}
{by using a result similar to Lemma \ref{lem:doubling_box}.}
Hence we complete the proof. 
\end{proof}

Now we note here that the second key estimate (Lemma \ref{lem:2bg_key}) easily follows from \eqref{eq:EE_consequence} and the Cauchy-Schwarz inequality in the same way as Lemma \ref{lem:one-block_rough_bound}. 
Hence the proof of the second-order Boltzmann-Gibbs principle is completed.

\section{Proof Outline: Derivation of KPZ/SBE}
\label{sec:pep_outline}
In this section, we give a sketch of the proof of Theorem \ref{thm:sbe_from_pep}.  
Our starting point is Dynkin's martingale decomposition. 
For $\varphi \in \mathcal{S}(\mathbb R)$, 
\begin{equation*}
\mathcal{M}^n_t(\varphi)
= \mathcal{X}^n_t(\varphi) 
- \mathcal{X}^n_0(\varphi)
- \int_0^t (\partial_s + L_n) \mathcal{X}^n_s(\varphi) ds
\end{equation*} 
and $\mathcal{M}^n_t(\varphi)^2 - \langle \mathcal{M}^n(\varphi) \rangle_t$ are martingales with respect to the natural filtration where 
\begin{equation}
\label{eq:pep_quadratic_variation}
\begin{aligned}
\langle \mathcal{M}^n(\varphi) \rangle_t 
&= \int_0^t \big( L_n(\mathcal{X}^n_s(\varphi))^2 - 2 \mathcal{X}^n_s(\varphi) L_n \mathcal{X}^n_s(\varphi) \big) ds \\
&=\int_0^t \frac{1}{n} \sum_{x\in\mathbb Z} 
\big( p_n r(\eta_x,\eta_{x+1}) 
+ q_n r(\eta_{x+1},\eta_x) \big)(s) 
T^-_{v_ns} (\nabla^n\varphi^n_x)^2 ds .
\end{aligned}
\end{equation} 
Since the measure $\nu_\rho$ is product, note that $E_{\nu_\rho}[r(\eta_x,\eta_y)]=E_{\nu_\rho}[r(\eta_y,\eta_x)]$ for any $x,y \in \mathbb Z$ such that $x\neq y$. 
Then, 
\begin{equation*}
\lim_{n\to \infty}
\langle \mathcal{M}^n(\varphi) \rangle_t 
= t E_{\nu_\rho}[r(\eta_0,\eta_1)]\| \nabla \varphi \|^2_{L^2(\mathbb R)}{=t \Phi_r(\rho) \| \nabla \varphi \|^2_{L^2(\mathbb R)}}
\end{equation*}
in $L^2(\mathbb P_n)$, which indicates that the martingale part converges to the space-time white-noise. 
Hereafter we compute the action of the generator, splitting it into symmetric and anti-symmetric parts. 
We begin with the symmetric part. 
Note that 
\begin{equation*}
\begin{aligned}
\int_0^t S_n\mathcal{X}^n_s(\varphi) ds 
&= \sqrt{n} \int_0^t \sum_{x\in\mathbb Z} 
W^S_{x,x+1}(s) T^-_{v_ns}\nabla^n\varphi^n_x ds \\
&= \frac{1}{2\sqrt{n}} \int_0^t \sum_{x\in\mathbb Z} 
c(\kappa) c(\eta_x) T^-_{v_ns} \Delta^n \varphi^n_x ds 
\end{aligned}
\end{equation*}
where in the last line we used the gradient condition \eqref{eq:gradient_condition} and executed the second integration-by-parts.  Note that since $\sum_{x\in\mathbb Z}\Delta_N \varphi^n_x=0$ we can introduce for free the average of $c(\eta_x)$ with respect to the invariant measure, namely $\Phi_c(\rho)$. 
Therefore, the last display writes as 
\begin{equation*}
\begin{aligned}
 \frac{1}{2\sqrt{n}} \int_0^t \sum_{x\in\mathbb Z} 
c(\kappa) \big\{c(\eta_x)-\Phi_c(\rho)\big\} T^-_{v_ns} \Delta^n \varphi^n_x ds 
\end{aligned}
\end{equation*}
Moreover, by the first-order Boltzmann-Gibbs principle, we can replace $c(\eta_x)-\Phi_c(\rho)$ by $\Phi_c'(\rho) \overline{\eta}_x$.
%$(d/d\sigma) E_{\nu_\sigma}[c(\eta_x)]|_{\sigma=\rho} (\eta_x-\rho)$. 
Hence the last quantity is further computed as 
\begin{equation*}
D(\rho)%\frac{d(\kappa)}{2} \frac{\partial}{\partial\rho} E_{\nu_\rho}[c(\eta_x)] 
\int_0^t \mathcal{X}^n_s(\Delta \varphi^n_x) ds
\end{equation*}
plus some error term which vanishes in the limit, using $D(\rho)$ defined in \eqref{eq:diffusion_coefficient}.  
Above, we replaced the discrete Laplacian by the continuous one, with a cost that can be estimated by the Cauchy-Schwarz inequality, and is of order 
\begin{equation*}
\begin{aligned}
&\mathbb E_n \bigg[ \sup_{0\le t\le T}
\bigg| \int_0^t \frac{1}{\sqrt{n}}
\sum_{x\in\mathbb Z} \Phi_c(\rho)\bar\eta_x
T^-_{v_ns} (\Delta^n \varphi^n_x - \Delta \varphi^n_x)
ds \bigg|^2 \bigg]\\
&\quad\le {\Phi_c(\rho)^2 \chi(\rho)}\frac{T}{n} \int_0^T \sum_{x\in\mathbb Z} 
T^-_{v_ns}(\Delta^n \varphi^n_x - \Delta \varphi^n_x)^2 ds
= O(n^{-4})
\end{aligned}
\end{equation*}
noting $|\Delta^n\varphi^n_x-\Delta \varphi^n_x| = O(n^{-2})$ by Taylor's theorem.
Next, we deal with the anti-symmetric part. 
Recall the velocity $v_n$ defined by \eqref{eq:pep_moving_frame} and the current $W^A_{x,x+1}$ given in \eqref{eq:pep_current_computation_anti-symmetric}. 
We note that 
\begin{equation*}
\begin{aligned}
\int_0^t (\partial_s + A_n) \mathcal{X}^n_s(\varphi) ds 
&= \sqrt{n} \int_0^t \sum_{x\in\mathbb Z} 
\Big( W^A_{x,x+1}(\eta(s)) -\frac{v_n}{n^2} \overline\eta_x (s)\Big) T^-_{v_ns}\nabla^n\varphi^n_x ds \\
&= \sqrt{n}(p_n-q_n) \int_0^t\sum_{x\in\mathbb Z} 
\tau_xV (\eta(s)) T^-_{v_ns} \nabla^n \varphi^n_x ds
\end{aligned}\end{equation*}
plus some small factor which comes from the replacement of the continuous derivative by the discrete one. 
Above, we defined the local function $V$ by 
\begin{equation}
\label{eq:nonlinear_term_element}
V(\eta)
= \frac{1}{2}(r_{0,1}+r_{1,0})(\eta) 
- \frac{d}{d\rho}(2\chi(\rho)D(\rho)) \overline{\eta}_0.  
\end{equation}
By \eqref{eq:gk_formula}, the local function $V$ satisfies 
\begin{equation*}
\frac{d}{d\sigma} E_{\nu_\sigma}[V(\eta)]\Big|_{\sigma=\rho}
= 0. 
\end{equation*}
%Moreover, note that the constant term is irrelevant in the limit. 
Therefore, according to the second-order Boltzmann-Gibbs principle (Proposition \ref{prop:2bg}), the anti-symmetric part gives rise in the limit to the nonlinear term of the SBE.  
In summary, we were able to obtain the following decomposition. 
\begin{equation}
\label{eq:martingale_decomposition}
\mathcal{X}^n_t(\varphi)
= \mathcal{X}^n_0(\varphi)
+ \mathcal{S}^n_t(\varphi)
+ \mathcal{B}^n_t(\varphi)
+ \mathcal{M}^n_t(\varphi)
+ \mathcal{R}^n_t(\varphi)
\end{equation}
where for each test function $\varphi\in\mathcal{S}(\mathbb R)$, we have  
\begin{equation*}
\mathcal{S}^n_t(\varphi)
= D(\rho) %\frac{d(\kappa)}{2} \frac{\partial}{\partial\rho} E_{\nu_\rho}[c(\eta_x)] 
\int_0^t \mathcal{X}^n_s(\Delta \varphi)ds 
\end{equation*}
and 
\begin{equation}
\label{eq:anti-symmetric_part_definition}
\mathcal{B}^n_t(\varphi)
= \alpha \int_0^t \sum_{x\in\mathbb Z} 
\tau_x\overline V(\eta(s)) T^-_{v_ns}\nabla^n\varphi^n_x ds ,
\end{equation}
with $\overline V= V-E_{\nu_\rho}[V]$ where the local function $V$ is defined in \eqref{eq:nonlinear_term_element}. 
Moreover, $\mathcal{R}^n_t(\varphi)$ is a negligible remainder term in the sense that 
\begin{equation*}
\lim_{n\to\infty}\mathbb E_n\bigg[\sup_{0\le t\le T} \big| \mathcal{R}^n_t(\varphi) \big|^2 \bigg]  
= 0. 
\end{equation*}

\section{Tightness}
\label{sec:pep_tightness}
In this section, we show that {the sequence $\{\mathcal{X}_t^n\}_n$} is tight. 
For that purpose, we apply Mitoma's criterion~\cite[Theorem 4.1]{mitoma1983tightness} to see that it suffices to show the tightness of the sequence of real-valued processes $\{ \mathcal{X}^n_t(\varphi)\}_n$ for any $\varphi\in\mathcal{S}(\mathbb R)$. 

{Recall that we have the martingale decomposition \eqref{eq:martingale_decomposition}. Note that for the initial field, by looking at the characteristic function, for each $\varphi\in \mathcal{S}(\mathbb R)$, the sequence $\{\mathcal{X}^n_0(\varphi)\}_n$, converges to a normal random variable, which particularly means that it is tight. 
Hence, we show the tightness of the martingale, and the time integral of symmetric and anti-symmetric parts. In the following, we also observe that all of our limiting objects are continuous functions of time. The tightness in the Skorohod space plus the continuity of the limit implies the tightness with respect to the uniform topology, for details, we refer the reader to \cite[Section 12, p.124]{billingsley1968convergence}.

Hence, the tightness with respect to the uniform topology of the fields $\{ \mathcal{X}^n_t(\varphi): t\in [0,T]\}_n$, and thus that of $\{ \mathcal{X}^n_t: t\in [0,T]\}_n$, follow from the martingale decomposition \eqref{eq:martingale_decomposition}. 

\subsection{Martingale Part}
To show the tightness of the martingale part, we use the following criterion given in \cite[Theorem VIII.3.12]{jacod2013limit}.

\begin{proposition}
\label{prop:martingale_characterization}
Let $\{M^n_t\}_{n\in\mathbb N}$ be a sequence of martingales belonging to the space $D([0,T],\mathbb R)$ and denote by $\langle M^n\rangle_\cdot$ the quadratic variation 
%\ms{(since this is the jump process, we may need to define more precisely what is the quadratic variation)}\textcolor{magenta}{Since this is taken from Jacod Shiryauev the reader can check there the precise definition, I would leave it like this}\ms{ok!} 
of $M^n_\cdot$, for any $n\in\mathbb N$. 
Let $c:[0,T]\to\mathbb R_+$ be a deterministic continuous function. 
Assume that 
\begin{itemize}
    \item[(1)] for any $n\in\mathbb N$, the quadratic variation process $\{\langle M^n\rangle_t:t\in[0,T]\}$ has continuous trajectories almost surely,
    \item[(2)] the following limit holds:
    \begin{equation*}
    \lim_{n\to\infty} \mathbb E\Big[\sup_{0\le t\le T} \big| M^n_t-M^n_{t-}\big| \Big] 
    = 0,
    \end{equation*}
    \item[(3)] for any $t\in[0,T]$, the sequence of random variables $\{\langle M^n\rangle_t \}_n$ converges to $c(t)$ in probability.
\end{itemize}
Then, the sequence $\{M^n_t:t\in[0,T]\}_n$ converges in distribution on $D([0,T],\mathbb R)$ as $n\to\infty$ to a mean-zero Gaussian process $\{ M_t :t\in[0,T]\}$, which is a martingale with continuous trajectories and whose quadratic variation is given by $c$. 
\end{proposition}

Once we could verify all items of the previous criterion, the convergence, and thus tightness, of the sequence $\{\mathcal{M}^n_t(\varphi):t\in[0,T]\}_n$ follow for any $\varphi\in\mathcal{S}(\mathbb R)$. 
The first item of Proposition \ref{prop:martingale_characterization} easily follows from the expression \eqref{eq:pep_quadratic_variation}. 
To show the second item, note that we have the relation 
\begin{equation*}
\sup_{0\le t\le T} \big| \mathcal{M}^n_t(\varphi) -\mathcal{M}^n_{t^-}(\varphi) \big|
\le 2 \sup_{0\le t\le T} \big| \mathcal{X}^n_t(\varphi) -\mathcal{X}^n_{t^-}(\varphi) \big|
\end{equation*}
for any $\varphi\in\mathcal{S}(\mathbb R)$. 
Suppose that a jump from $x$ to $x+1$ occurs at time $t^-$. 
Noting the probability that more than two jumps occur simultaneously is zero, we have that
\begin{equation*}
\big| \mathcal{X}^n_t(\varphi) -\mathcal{X}^n_{t-}(\varphi) \big|
= \frac{1}{\sqrt{n}} | \varphi^n_{x+1}-\varphi^n_x|
\le n^{-3/2} \|\nabla \varphi\|_\infty. 
\end{equation*}
%\ms{(I would like to confirm : Is it clear that "the probability that more than two jumps occur simultaneously is zero" for a process on $\mathbb Z$? I believe this should be true though...)}\textcolor{magenta}{Yes, I don't see what could be the problem, at each time there is at most one jump.}\ms{Ok, I checked Ligget's book and agreed. }
Since the right-hand side of the last display does not depend on $t$, the second item of Proposition \ref{prop:martingale_characterization} holds. 
Finally, regarding the third item, recall the expression \eqref{eq:pep_quadratic_variation}. 
Even more strongly than the condition, we can show that 
\begin{equation*}
\lim_{n\to\infty}
\mathbb E_n \Big[ \Big( 
\langle \mathcal{M}^n(\varphi)\rangle_t
 -E_{\nu_\rho}[\langle \mathcal{M}^n(\varphi)\rangle_t]
\Big)^2 \Big]
=0 
\end{equation*}
by a direct computation. 
Then, by combining this with the fact that $\langle \mathcal{M}^n(\varphi)\rangle_t$ converges to $t\Phi_r(\rho) \|\nabla \varphi\|^2_{L^2(\mathbb R)}$ in $L^2{(\mathbb P_n)}$, we could verify the third item and thus the tightness of the martingale part follows.

\subsection{Symmetric Part} 
Note that by the Cauchy-Schwarz inequality we have that 
\begin{equation*}
\begin{aligned}
\mathbb E_n \Big[ \big| \mathcal{S}^n_t(\varphi)-\mathcal{S}^n_s(\varphi)\big|^2 \Big]
&\le C |t-s| \int_s^t \frac{1}{n} \sum_{x\in\mathbb Z} 
E_{\nu_\rho}[(\overline\eta_x)^2] 
(T^-_{v_nr}\Delta \varphi^n_x)^2 dr \\
&\le C |t-s|^2 \| \Delta \varphi\|^2_{L^2(\mathbb R)}
\end{aligned}
\end{equation*}
for any $s,t\in [0,T]$.
Hence the tightness of the symmetric part follows from the Kolmogorov-Chentsov criterion (see for example \cite[Theorem 1.2.1]{revuz2013continuous}) and continuity of the process.

\subsection{Anti-Symmetric Part}
Finally, we conclude this section by showing the tightness of the anti-symmetric part. 
By the second-order Boltzmann-Gibbs principle (Proposition \ref{prop:2bg}) and stationarity, there exists a constant $C>0$ such that  
\begin{equation*}
\begin{aligned}
& \mathbb E_n\bigg[ \bigg| 
\mathcal{B}^n_t(\varphi)-\mathcal{B}^n_s(\varphi)
- \Lambda(\rho) 
%\frac{\alpha}{2} \frac{\partial^2}{\partial\rho^2} E_{\nu_\rho}[V]
\int_s^t \sum_{x\in\mathbb Z} \big( \overrightarrow\eta^\ell_x(r)-\rho\big)^2 
T^-_{v_nr}\nabla^n \varphi^n_x dr \bigg|^2 \bigg]  \\
&\quad \le C \Big(\frac{(t-s)\ell}{n}+ \frac{(t-s)^2n}{\ell^2} \Big) \| \nabla\varphi\|^2_{L^2(\mathbb R)}
\end{aligned}
\end{equation*}
for any $s,t \in [0,T]$ such that $s<t$.
Above, recall that $\Lambda(\rho)$ is defined in \eqref{eq:sbe_nonlinear_coeff}. 
On the other hand, note that by a direct $L^2$-estimate, we have that 
\begin{equation*}
\mathbb E_n \bigg[ \bigg| \int_s^t \sum_{x\in\mathbb Z} \big( \overrightarrow\eta^\ell_x(r)-\rho\big)^2 
T^-_{v_nr}\nabla^n \varphi^n_x dr\bigg|^2 \bigg] 
\le C \frac{(t-s)^2n}{\ell} \|\nabla\varphi\|^2_{L^2(\mathbb R)}. 
\end{equation*}
When $1/n^2 \le t-s$, we take a natural number $\ell$ {as $\ell=[(t-s)^{1/2}n]$ where $[\cdot]$ is the Gauss symbol}, which yields 
\begin{equation*}
\mathbb E_n \Big[\big| \mathcal{B}^n_t(\varphi)-\mathcal{B}^n_s(\varphi)\big|^2 \Big] 
\le C (t-s)^{3/2} \|\nabla\varphi\|^2_{L^2(\mathbb R)}. 
\end{equation*}
On the other hand, when $t-s \le 1/n^2$, {recalling} {from \eqref{eq:anti-symmetric_part_definition}} {the definition of $\mathcal{B}^n_t(\varphi)$} we can directly estimate as 
\begin{equation*}
\mathbb E_n 
\Big[\big| \mathcal{B}^n_t(\varphi)-\mathcal{B}^n_s(\varphi)\big|^2 \Big] 
\le C (t-s)^2 n \| \nabla\varphi\|^2_{L^2(\mathbb R)} 
\le C (t-s)^{3/2} \| \nabla\varphi\|^2_{L^2(\mathbb R)} . 
\end{equation*}
Hence, similarly to the symmetric part, the tightness of the anti-symmetric part follows from the Kolmogorov-Chentsov criterion and the continuity of the process.

\section{Identification of Limit Points}
\label{sec:pep_identification}
Up to now, we proved that each fluctuation field appearing in the martingale decomposition \eqref{eq:martingale_decomposition} is tight with respect to the Skorohod topology of $D([0,T],\mathcal{S}'(\mathbb R))$. 
In this section we identify the limit points of these sequences and complete the proof of Theorem \ref{thm:sbe_from_pep}. 
Let $\mathscr Q^n$ be the distribution of 
\begin{equation*}
\big\{(\mathcal{X}^n_t, \mathcal{M}^n_t, \mathcal{S}^n_t, \mathcal{B}^n_t)
: t\in [0,T] \big\}  
\end{equation*}
for each $n \in\mathbb N$.
Then, there exists a subsequence $n$, which is denoted by the same letter by abuse of notation, such that the sequence $\{ \mathscr Q^n\}_n$ converges to a limit point $\mathscr Q$. 
Let $\mathcal{X}$, $\mathcal{M}$, $\mathcal{S}$ and $\mathcal{B}$ be the respective limits in distribution of each component. 
Since the tightness is shown also with respect to the uniform topology of $D([0,T], \mathcal{S}'(\mathbb R))$, we know that all these limiting processes almost surely have continuous trajectories. 
Therefore, it suffices to characterize the limit points by the stationary energy solution of the SBE in the sense of Definition \ref{def:energysol}, since the convergence along the full sequence $n$ follows by the uniqueness in law of the solution. 

First, for the characterization of the martingale part, recall that we already proved, by checking the conditions of Proposition \ref{prop:martingale_characterization}, that the limit point $\{\mathcal{M}_t(\varphi): t\in[0,T]\}$ is a continuous martingale whose quadratic variation is $t \|\nabla\varphi\|^2_{L^2(\mathbb R)}$. 
%This can be conducted by a routine work, so that we omit the proof here, see \cite[Subsection 5.3.1]{goncalves2023derivation} for the details. 
For the symmetric part, we can easily show that any limit point $\mathcal{S}$ satisfies 
\begin{equation*}
\mathcal{S}_t (\varphi)
= D(\rho)\int_0^t \mathcal{X}_s(\Delta \varphi)ds 
\end{equation*}
in distribution, for any $\varphi\in\mathcal{S}(\mathbb R)$. 
Therefore, our task is now to characterize the limit point of the anti-symmetric part.
Note here that we have the identity 
\begin{equation*}
\mathcal{X}^n_t\big( \iota_\varepsilon\big(
{\textstyle \frac{x-v_nt}{n}, \cdot} \big)\big)
= \sqrt{n} \big( \overrightarrow\eta^{\varepsilon n}_x-\rho\big)
\end{equation*}
where recall that the function $\iota_\varepsilon$ is defined in \eqref{eq:definition_iota}. 
Using this identity, recall that the second-order Boltzmann-Gibbs principle yields
\begin{equation}
\label{eq:2bg_consequence}
\begin{aligned}
& \mathbb E_n\bigg[ \bigg| 
\mathcal{B}^n_t(\varphi)-\mathcal{B}^n_s(\varphi)
- \frac{\Lambda(\rho)}{n}
%\frac{\alpha}{2n} \frac{\partial^2}{\partial\rho^2} E_{\nu_\rho}[r_{0,1}] 
\int_s^t \sum_{x\in\mathbb Z} \mathcal{X}^n_t\big( \iota_\varepsilon\big(
{\textstyle \frac{x-v_nr}{n}, \cdot} \big)\big)^2 
T^-_{v_nr}\nabla^n \varphi^n_x dr \bigg|^2 \bigg]  \\
&\quad \le C \Big(\frac{(t-s)\ell}{n}+ \frac{(t-s)^2n}{\ell^2} \Big) \| \nabla\varphi\|^2_{L^2(\mathbb R)}.
\end{aligned}
\end{equation}
Moreover, note that the limit point $\mathcal{X}$ of the sequence $\{ \mathcal{X}^n_t: t\in [0,T]\}_n$ clearly satisfies the condition \textbf{(S)}. 
Consequently, we obtain the limit 
\begin{equation*}
\mathcal{A}^\varepsilon_{s,t}(\varphi)
= \lim_{n\to\infty} \frac{1}{n}
\int_s^t \sum_{x\in\mathbb Z} \mathcal{X}^n_t\big( \iota_\varepsilon\big(
{\textstyle \frac{x-v_nt}{n}, \cdot} \big)\big)^2
T^-_{v_nr}\nabla^n \varphi^n_x dr 
\end{equation*}
where $\mathcal{A}^\varepsilon_{s,t}(\varphi)$ is the process we defined in \eqref{eq:def_quadratic_function_approximation} with $u= \mathcal{X}$.   
Here note that the above convergence does not follow immediately since the function $\iota_\varepsilon$ is not in the space $\mathcal{S}(\mathbb R)$, but this procedure can justified by a proper approximation, see \cite[Section 5.3]{gonccalves2014nonlinear} for details. 
Hence, {letting $\ell=\varepsilon n$} and $n\to \infty$ in \eqref{eq:2bg_consequence}, we obtain 
\begin{equation*}
\mathbb E \Big[ \Big| \mathcal{B}_t(\varphi)-\mathcal{B}_s(\varphi) 
- \Lambda(\rho) %\frac{\alpha}{2} \frac{\partial^2}{\partial\rho^2}E_{\nu_\rho}[r_{0,1}] 
\mathcal{A}^\varepsilon_{s,t} (\varphi) \Big|^2\Big] 
\le C \varepsilon (t-s) \| \nabla\varphi\|^2_{L^2(\mathbb R)}.
\end{equation*}
By the triangle inequality, the energy condition \textbf{(EC)} follows. 
As a result, by Proposition \ref{prop:nonlinear}, we obtain the existence of the limit 
\begin{equation*}
\mathcal{A}_t(\varphi) 
= \lim_{\varepsilon\to 0} \mathcal{A}^\varepsilon_{0,t}(\varphi)
\end{equation*}
and the above consequence of the second-order Boltzmann-Gibbs principle yields that 
\begin{equation*}
\mathcal{B} =
\Lambda(\rho) 
%\frac{\alpha}{2} \frac{\partial^2}{\partial\rho^2}E_{\nu_\rho}[r_{0,1}] 
\mathcal{A} .  
\end{equation*}

Finally, we note here that all the above estimates hold in exactly the same way for the reversed process $\{ \mathcal{X}^n_{T-t}:t\in [0,T]\}$ by repeating the same argument for the dynamics generated by $L_n^*$, so that the third item in Definition \ref{def:energysol} is satisfied. 
Hence, we conclude that the process $\mathcal{X}$ is the stationary energy solution of the SBE and complete the proof of Theorem \ref{thm:sbe_from_pep}.

\section*{Acknowledgments}
%KH was supported by JSPS KAKENHI Grant Number JP22J12607.
%Additionally, the author would like to thank Makiko Sasada and Hayate Suda for giving him fruitful comments and suggestions. 
P.G. thanks FCT/Portugal for financial support
through the research center CAMGSD, IST-ID, projects UIDB/04459/2020 and UIDP/04459/2020 and also for financial support through the project ERC-FCT SAUL. 
MS was supported by JSPS KAKENHI Grant Number JP18H03672.

%%%%%%%%%%%%%%%%
%  References  %
%%%%%%%%%%%%%%%%
\bibliographystyle{abbrv}
\bibliography{ref}

\end{document}